\documentclass[12pt]{article}
\usepackage{amsmath,amssymb,amsthm,bm}
\usepackage{enumitem}
\usepackage{color}
\usepackage{tikz-cd}
\usepackage{url}
\usepackage{url}

\usepackage[backend=biber]{biblatex}
\theoremstyle{plain}
\newtheorem{theorem}{Theorem}[section]
\newtheorem{lemma}[theorem]{Lemma}
\newtheorem{corollary}[theorem]{Corollary}
\newtheorem{definition}[theorem]{Definition}
\newtheorem{example}[theorem]{Example}

\newcommand{\Z}{\mathbb{Z}}

\newcommand{\K}{\mathbb{K}}

\newcommand{\N}{\mathbb{N}}

\newcommand{\mx}{\bm{x}}

\newcommand{\Kx}{\K[\bm{x}]}

\newcommand{\mal}{\bm{\alpha}}

\newcommand{\mbe}{\bm{\beta}}

\newcommand{\mga}{\bm{\gamma}}

\newcommand{\mde}{\bm{\delta}}

\newcommand{\Ka}{\K_{\mal}}
\newcommand{\Kab}{\,_{\mal}\K_{\mbe}}

\DeclareMathOperator{\Der}{Der}

\DeclareMathOperator{\Sp}{ssp}

\DeclareMathOperator{\Lt}{lt}
\DeclareMathOperator{\lt}{lt}
\DeclareMathOperator{\Lc}{lc}
\DeclareMathOperator{\lc}{lc}
\DeclareMathOperator{\Lm}{lm}
\DeclareMathOperator{\lm}{lm}

\DeclareMathOperator{\Ker}{ker}

\DeclareMathOperator{\Frac}{Frac}
\DeclareMathOperator{\codim}{codim}

\newcommand\restr[2]{{
  \left.\kern-\nulldelimiterspace 
  #1 
  \vphantom{\big|} 
  \right|_{#2} 
  }}

\title{Describing Multivariate Polynomial Subalgebras Using Equations}
\author{ Erik Leffler\thanks{Lund University}}
\date{}
\begin{document}
\maketitle

\begin{abstract}
	In this paper we continue the work of describing polynomial subalgebras of
	finite codimension that was started in \cite{gronkvist2022subalgebras}. Let
	$\K$ be an algebraically closed field, and $A \subset \K[x_{1}, \ldots,
			x_n]$ be a subalgebra of finite codimension. It is known that there exists
	a (not necessarily unique) finite filtration of $\K$-algebras
	\[
		A = A_{0} \subset A_{1} \subset \ldots \subset A_m = \K[x_{1}, \ldots, x_n],
	\]
	where each $A_i$ can be written as the kernel of some linear functional
	$L_{i + 1} : A_{i + 1} \to \K$, and each $L_i$ is either a derivation or of
	the form $L_i : f \to c(f(\mal) - f(\mbe))$ for some $\mal, \mbe \in
		\K^{n}$ and $c \in \K$. We investigate the structure of these filtrations
	and linear functionals. Our main result shows that each such $L_i$ which is
	a derivation may be written as a linear combination of partial derivatives
	evaluated at points of $\K^{n}$.
	
\end{abstract}

\tableofcontents

\section{Introduction}
\label{intro}

Let $\K$ be an algebraically closed field of characteristic $0$, and $\mx =
\{x_1,x_2, \ldots, x_n\}$ be a set of $n$ indeterminates. Throughout this text
we shall be concerned with polynomial subalgebras $A \subseteq \Kx$ of finite
codimension. We will require that our subalgebras are unital and thus contain
the scalar field $\K \subset A$.  Usually such subalgebras would be described
in terms of a generating set. An example in the univariate case is the
subalgebra $A_1 \subset \K[x]$ generated by the polynomials $x^3, x^4$ and
$x^5$. Then $A_1$ consists of all polynomials which contain no first or second
degree terms. In \cite{gronkvist2022subalgebras}, a theory is developed to
describe such subalgebras using defining conditions. For example, the same
subalgebra $A_1$ can be written using conditions as follows,
$$
	A_1 = \{f \in \Kx\ |\ f'(0) = f''(0) = 0\}.
$$
Another example: let $A_2 = \K\left[x^3 - x, x^2 \right]$ be the smallest
algebra containing $x^3 - x, x^2$, and $\K$. This algebra can be written using
conditions as
$$
	A_2 = \{f \in \Kx\ |\ f(1) = f(-1)\}.
$$
The theory developed in \cite{gronkvist2022subalgebras} is limited to
univariate subalgebras of finite codimension. This paper generalizes most of
those results to the multivariate case. We remain limited to subalgebras of
finite codimension. \\

A multivariate example is given by the algebra
\[
	A_3
	=
	\K\left[
	x_{2},\:
	x_{1}^{2},\:
	x_{1}x_{2}^{2} - 2\,\:
	x_{1}x_{2} + x_{1},\:
	x_{1}^{3}
	\right],
\]
which can be described by conditions as 
\[
	A_3
	=
	\{
	f \in \Kx 
	:
	f'_{x_1}(0, 1) = f''_{x_1,x_2}(0, 1) = 0
	\}.
\]
At this point, it may be difficult to see that these descriptions are
equivalent, but it will become clear once we have introduced the necessary
theory. \\ 

The objective of this paper is to better understand subalgebras defined by
conditions, and the structure of such conditions. An immediate consequence of
our main result given in Theorem \ref{th-main}, is that every subalgebra of
finite codimension in $\Kx$ can be written as the set of polynomials $f \in
	\Kx$ satisfying a finite set of conditions, where each condition is either of
the type $f(\mal) = f(\mbe)$ for some $\mal, \mbe \in \K^{n}$, or a linear
relation on the partial derivatives of $f$ evaluated at points of $\K^{n}$. \\

Unfortunately, subalgebras of finite codimension often require generating sets
which are too large for us to interpret. Therefore, most examples of
subalgebras in this paper will be defined using conditions. The problem rapidly
grows worse as we increase either the number of indeterminates or the
codimension of the subalgebras we consider. Above we saw an example for a
subalgebra of codimension $2$ in $\K[x_1,x_2]$. To illustrate the issue, we
include a computer-generated example of a subalgebra of codimension $3$ in
$\K[x_1,x_2,x_3]$. Let $A_4 \subset \K[x_1,x_2,x_3]$ be defined by conditions
as  
\begin{align*}
	A_4
	=
	\{
	f \in \Kx 
	:
	 & f'_{x_3}(1, 0, -1) = 0,
	f(3, 2, 5) = f(1, -3, 2),   \\
	 & f'_{x_1}(3,2,5)
	-
	3f'_{x_2}(1, -3, 2) = 0
	\}.
\end{align*}
Then we can describe $A_4$ via a generating set as 

\begingroup
\allowdisplaybreaks
\begin{align*}
	A_4
	=
	\K\Bigg[ 
	 & y_{3}^{2}-\frac{81}{11}y_{1}-\frac{27}{11}y_{2}+2\,y_{3},\quad
	y_{2}y_{3}-\frac{18}{11}y_{1 }-\frac{28}{11}y_{2},\quad
	y_{1}y_{3}-5\,y_{1}-y_{3},                                           \\
	 & y_{2}^{2}-\frac{75}{11}y_{1}+\frac{41}{11}y _{2},\quad
	y_{1}y_{2}-2\,y_{1}-y_{2},\quad
	y_{1}^{2}-\frac{54}{11}y_{1}+\frac{4}{11}y_{2},                      \\
	 & y_{3}^{3}-\frac{324}{11}y_{1}-\frac{108}{11}y_{2}-3\,y_{3},\quad
	y_{2}y_{3}^{2}-\frac{126}{11}y_{1}-\frac{86}{ 11}y_{2},\quad
	y_{1}y_{3}^{2}-\frac{356}{11}y_{1}-\frac{27}{11}y_{2}+2\,y_{3},      \\
	 & y_{2}^{2}y_{3}-\frac{186}{11}y_{1}+\frac{70}{11}y_{2},\quad
	y_{1}y_{2}y_{3}-\frac{128}{11}y_{1}-\frac{28}{11}y_{2 },\quad
	y_{1}^{2}y_{3}-\frac{270}{11}y_{1}+\frac{20}{11}y_{2}-y_{3},         \\
	 & y_{2}^{3}+\frac{300}{11}y_{1 }-\frac{197}{11}y_{2},\quad
	y_{1}y_{2}^{2}-\frac{119}{11}y_{1}+\frac{41}{11}y_{2},\quad
	y_{1}^{2}y_{2 }-\frac{108}{11}y_{1}-\frac{3}{11}y_{2},               \\
	 & y_{1}^{3}-\frac{213}{11}y_{1}+\frac{28}{11}y_{2}
	\Bigg].
\end{align*}
\endgroup

The algorithm used to produce the generating set is based on Theorem
\ref{th-build_sagbi}, and this algorithm does not necessarily produce a minimal
generating set. It does however produce a minimal SAGBI basis, a concept which
will be explained shortly. 

\subsection{Notations and Conventions}
\label{Conventions}

We present some notations and conventions we will use before proceeding. \\ 

We adopt the convention that the set of natural numbers $\N$ includes $0$. \\ 

Instead of writing $\K[x_1,x_2, \ldots, x_n]$ we shall simply write $\Kx$.
Sometimes, instead of writing $x_1^{a_1}x_2^{a_2}\ldots x_n^{a_n}$ we will
write $\bm{x}^{\bm{a}}$ where $\bm{a} = (a_1, a_2, \ldots, a_n)$. Moreover, the
variable $n$ will always be used to denote the number of free variables $\bm{x}
	= \{x_1,x_2,\ldots,x_n\}$ in our algebra. \\

Given some term ordering, we will use $\lm(f), \lt(f), \lc(f)$ to denote the
leading monomial, leading term, and leading coefficient of a polynomial $f \in
\Kx$. Moreover, $\deg(f)$ will be used to denote the multidegree of $f$, and we
shall simply call this the degree of $f$. If $S \subset \Kx$ is a set of
polynomials, we will abuse notation and write $\Lm(S), \Lt(S), \Lc(S), \deg(S)$
to denote the sets obtained by applying the corresponding functions to all
elements of $S$. \\

Given a set $G$ and a multiplicative operation on the elements of $G$, we
denote by $G_{\text{mon}}$ the set of all finite products of elements of $G$.
\\

Given a vector space $V$ and a subset $S \subset V$, we denote by $\langle S
\rangle$ the subspace of $V$ generated by $S$. \\

We write the range of integers $\{x \in \Z : l \leq x \leq u\}$ as $[l..u]$.
Note the inclusive bounds. \\

Let $\mal \in \K^{n}$ be a point. Then $\mathfrak{m}_{\mal}$ will denote the
maximal ideal in $\Kx$ of polynomials vanishing at $\mal$, and
$\mathfrak{m}_{\mal}(A) = \mathfrak{m}_{\mal} \cap A$ will denote the
restriction of this ideal to $A$ whenever $A \subset \Kx$ is a subalgebra. We
will write $\Ka = \Kx/\mathfrak{m}_{\mal}$ for the $\Kx$-module which is equal
to $\K$ as a set, but where polynomials $f \in \Kx$ act on scalars $c \in \Ka$
via $f \cdot c = f(\mal)c$. \\

Whenever $A$ is a $\K$-algebra, a $\K$-algebra morphism $L : A \to \K$ will be
called a character.

\section{Background}
\label{background}

\subsection{SAGBI Basis}
\label{sagbi}

We will use SAGBI bases occasionally throughout the text, and give a brief
definition here. For a proper treatment of SAGBI bases, we refer to
\cite{robbianosweedler1990sagbi}.

\begin{definition}
	Let $A \subseteq \Kx$ be a subalgebra. A subset $G \subseteq A$ is said to
	be a SAGBI basis for $A$ if $\deg(G)$ generates $\deg(A)$ as a semigroup,
	or equivalently, if $\Lm(G)_{\text{mon}} = \Lm(A)$.
\end{definition}

One of the main benefits of SAGBI bases is that they enable the subduction
algorithm, which can be used to test subalgebra membership. We give a quick
demonstration: Let $f \in \Kx$. If $\lm(f) \in \Lm(G)_{\text{mon}}$, then there
exists a scalar $a_0 \in \K$ and product of elements in $G$, call it $h_0 =
\prod_{g_i \in G}g_i^{b_i}$, such that $a_0h_0$ has the same leading term as
$f$. Then $\lm(f - a_0 h_0) < \lm(f)$. Set $f_0 = f - a_0h_0$. If $\lm(f_0) \in
\Lm(G)_{\text{mon}}$, we repeat the process above to obtain a polynomial $f_1 =
f_0 - a_1h_1$ with $\lm(f_1) < \lm(f_0)$. We continue this process until we
reach a point where either $f_m = 0$, or $\lm(f_m) \not \in
\Lm(G)_{\text{mon}}$. We say that $f$ subduces to $f_m$ over $G$. When $G$ is a
SAGBI basis for $A$, $f$ subduces to $0$ if and only if $f \in A$. \\ 

In general, a subalgebra $A \subset \Kx$ need not admit a finite SAGBI basis. 
Thankfully, subalgebras $A \subset \Kx$ of finite codimension always do.
\begin{lemma}\label{lm-finit_SAGBI}
    Let $A \subset \Kx$ be a subalgebra of finite codimension. Then $A$ admits
    a finite SAGBI basis.
\end{lemma}
\begin{proof}
    Follows from Proposition 4.7 and 4.9 in \cite{robbianosweedler1990sagbi}.
    
\end{proof}

\subsection{Linear Functional Lemma}

Computing kernels of sets of linear functionals will be central throughout this
text, and the following elementary lemma will prove useful.

\begin{lemma}\label{lm-intersect_lin_funcs}
	Let $V$ be a vector space over the field $\mathbb{K}$ and let $L_i : V
		\rightarrow \mathbb{K}$ for $i \in [1..m]$ be linear functionals. If we
	denote 
	\[
		W = \bigcap_{i = 1}^m \Ker(L_i),
	\]
	and $L : V \rightarrow \K$ is a linear functional on $V$ such that
	$\restr{L}{W} = 0$. Then
	\[
		L = \sum_{i = 1}^m  c_i L_i
	\]
	for some set of scalars $c_i \in \K$.
\end{lemma}
\begin{proof}
	We prove this by induction on $m$. \\
	
	Suppose that $m = 1$. The result is trivial whenever $L$ or $L_1$ are
	trivial, so assume that they are not. Let $v_1 \in V$ be such that
	$L_1(v_1) = 1$, and let $v \in V$. As $L_1(v - L_1(v)v_1) = 0$, we have $v
		= L_1(v) v_1 + u$ for some $u \in \Ker(L_1)$. As $L$ vanishes on
	$\Ker(L_1)$ we get $L(v) = L_1(v)L(v_1)$, and the statement of the lemma
	holds with $c_1 = L(v_1)$. \\
	
	Now suppose that the statement of the lemma holds for $m-1$ linear
	functionals. Let $W_{m - 1} = \bigcap_{i = 1}^{m - 1} \Ker(L_i)$. Then
	$\restr{L}{W_{m - 1}}$ vanishes on $\Ker(\restr{L_m}{W_{m - 1}})$, so
	$\restr{L}{W_{m - 1}} = c_m \restr{L_m}{W_{m - 1}}$ for some scalar $c_m
		\in \K$. Thus $L - c_m L_m$ vanishes on $W_{m - 1}$, and the inductive
	hypothesis yields $L = \sum_{i = 1}^m  c_i L_i$. 
	
\end{proof}

\section{Subalgebra Conditions}

This section will develop definitions of different kinds of subalgebra
conditions.

\begin{definition}
    Let $A$ be a $\K$-algebra and $L : A \to \K$ be a linear functional. We say
    that $L$ is a subalgebra condition on $L$ whenever $\ker(L)$ is a
    subalgebra of $A$.
\end{definition}

We refer to a finite set $\mathcal{L}$ of $|\mathcal{L}| = m$ linear
functionals $A \to \K$ as a set of subalgebra conditions whenever we can label
the functionals $\{L_{1}, \ldots, L_m\} = \mathcal{L}$ in such a way that 
\[
	A_i = \bigcap_{k = 1}^{m} \ker(L_k)
\]
is a subalgebra for all $i$. Thus a set of $m$ subalgebra conditions induces a
filtration of subalgebras
\[
	A_{m} \subset A_{m - 1} \ldots \subset A_{0} = A,
\]
where for each $i$, $L_i$ is a subalgebra condition on $A_{i - 1}$. Note that
we do not require the $L_i$ to be subalgebra conditions on $A$. \\

If $\mathcal{L}$ is a set of subalgebra conditions, we define the kernel of
$\mathcal{L}$ as $\Ker(\mathcal{L}) = \bigcap_{L \in \mathcal{L}} \Ker(L)$. The
definitions ensure that the kernel of a set of subalgebra conditions is a
subalgebra. \\

We will see in Theorem \ref{th-gorin} that every subalgebra of codimension $m$
in $\Kx$ can be written as the kernel of a set of $m$ subalgebra conditions,
each of which is one of two types to be defined below: a character difference,
or an $\mal$-derivation.

\subsection{Character Differences}

\begin{definition} 
	Let $\mal, \mbe \in \K^n, \mal \not = \mbe$ and $c \in \K$. We define a
	character difference to be a function $E : \Kx \to \K$ of the kind $E(f) =
		c (f(\mal) - f(\mbe))$. When we want to emphasize the points of evaluation,
	we call $E$ an $\mal, \mbe$-character difference. 
\end{definition}

It is straightforward to verify that the kernel of a character difference is
indeed an algebra. It's also easy to see that the set of $\mal, \mbe$-character
differences form a vector space. 

\subsection{$\mal$-Derivations}

We will define the second type of subalgebra condition as a special kind of
derivation. We recall the definition of a $\K$-linear derivation below for
convenience, and refer to \cite{eisenbud} for a thorough treatment.

\begin{definition}
	Let $A$ be a $\K$-algebra and $M$ be an $A$-module. A $\K$-linear
	function $D : A \to M$ is said to be a derivation if it satisfies the
	Leibniz rule,
	\[
		D(fg) = f \cdot D(g) + g \cdot D(f)
	\]
	for all $f, g \in A$. The set of $\K$-linear derivations $D : A \to M$ will
	be denoted by $\Der_{\K}(A, M)$, and this set forms an $A$-module where $f
		\in A$ acts on $D \in \Der_{\K}(A, M)$ via $(f \cdot D) : g \mapsto f \cdot
		D(g)$.
\end{definition}

Now we define the second type of subalgebra condition, which is obtained by
letting $M = \Ka$ in the definition above.

\begin{definition}
	Let $A \subset \Kx$ be a subalgebra and $\mal \in \K^{n}$. A $\K$-linear
	derivation from $A$ to $\Ka$ is called an $\mal$-derivation. The $A$-module
	$\Der_{\K}(A, \Ka)$ of all $\mal$-derivations on $A$ is called the
	$\mal$-derivation space, and we will write it as $\mathcal{D}_{\mal}(A)$.
\end{definition}
Note that $\mal$-derivations are precisely the linear functionals $D : A \to
\K$ which for all $f, g \in A$ satisfy
\[
	D(fg) = f(\mal)D(g) + D(f)g(\mal).
\]

As the kernel of any derivation is a subalgebra, it is immediate that
$\mal$-derivations are indeed subalgebra conditions. \\

The definition above is implicit in the sense that it does not tell us how to
construct $\mal$-derivations. This is a non-trivial issue that will be
addressed progressively throughout the paper, and our main contribution to this
end is the characterisation of $\mal$-derivations given in Theorem
\ref{th-main}. For now, we can give a couple of examples. \\

\begin{example}
	Let $A = \Kx$. Then it is easy to check that $D : A \to \K$ given by
	\[
		D(f) = \sum_{i = 1}^{n} a_i f'_{x_i}(\mal)
	\]
	is an $\mal$-derivation over $A$ for all choices of scalars $a_i$. If we now construct the algebra
	$A' = \Kx \cap \Ker\left(f \mapsto f'_{x_1}(\mal)\right)$, then 
	\[
		f 
		\mapsto 
		b_2 f''_{x_1 x_1}(\mal) + b_3 f'''_{x_1x_1 x_1}(\mal) + \sum_{i=2}^n
		a_i f'_{x_i}(\mal)
	\]
	is an $\mal$-derivation over $A'$ but not over $A$. 
	In a similar fashion,
	\[
		f \mapsto b_1 f^{(4)}_{x_1 x_1 x_1 x_1}(\mal) + \sum_{i=2}^n a_i f'_{x_i}(\mal)
	\]
	is an $\mal$-derivation over
	\[
		\Kx 
		\cap 
		\Ker\left(f \mapsto f'_{x_1}(\mal)\right) 
		\cap 
		\Ker\left(f \mapsto f_{x_1x_1}''(\mal)\right),
	\]
	but not over $A$ or $A'$. 
\end{example}

We will see later in the Main Theorem of $\mal$-Derivations (Theorem
\ref{th-main}) that all $\mal$-derivations can be expressed as linear
combinations of derivative evaluations as in the examples above.

\subsection{Unification of $\mal$-Derivations and Character Differences}

There is a perspective which unifies the definition of character differences
with that of $\mal$-derivations. To observe this, we need to consider
subalgebra conditions in the more general context of bimodule derivations. 
\begin{definition}
	Let $A$ be a $\K$-algebra, $M$ an $A$-bimodule and $D : A \to M$ be a
	$\K$-linear function. Then $D$ is a $\K$-linear bimodule derivation if it
	satisfies the Leibniz rule
	\[
		D(fg) = f \cdot D(g) + D(f) \cdot g,
	\]
	where we now require that $f$ acts on $D(g)$ on the left, and $g$
	acts on $D(f)$ on the right. 
\end{definition}

Now let $\Kab$ denote the $\Kx$-bimodule which is equal to $\K$ as a set, but
where $f \in \Kx$ acts on $c \in \K$ on the left via $f \cdot c = f(\mal)c$,
and on the right via $c \cdot f = f(\mbe)c$. Note that $\Kab \cong \Ka
\otimes_\K \K_{\mbe}$ as $\Kx$-bimodules. It turns out that the $\mal,
\mbe$-character differences on $A$ are precisely the $\K$-linear bimodule
derivations $A \to \Kab$.

\begin{lemma}
	Let $A \subset \Kx$ be a subalgebra and $\mal \not = \mbe \in \K^{n}$ be
	points such that there exists $h \in A$ where $h(\mal) \not = h(\mbe)$. Then
	the $\mal, \mbe$-character differences on $A$ are precisely the $\K$-linear
	bimodule derivations $A \to \Kab$.
\end{lemma}
\begin{proof}
	First let $L$ be an $\mal, \mbe$-character difference and $f, g \in A$.
	Then $L$ is clearly $\K$-linear, and it is straightforward to verify that
	it satisfies the Leibniz rule $A \to \Kab$:
	\begin{align*}
		L(fg) 
		 & =
		f(\mal)g(\mal) - f(\mbe)g(\mbe) \\
		 & =
		f(\mal)g(\mal) 
		- f(\mal)g(\mbe)
		+
		f(\mal)g(\mbe)
		- f(\mbe)g(\mbe)                \\
		 & =
		f(\mal)L(g) + g(\mbe)L(f).
	\end{align*}
	
	For the other direction, let $L \not = 0$ be a non-zero $\K$-linear
	bimodule derivation $A \to \Kab$. Pick $g \in A$ such that $L(g) = 0$. As
	$L$ is non-zero, we can pick $f \in A$ such that $L(f) \not = 0$. By
	comparing $L(fg)$ and $L(gf)$ we get
	\begin{align*}
		0
		 & =
		L(fg) - L(gf) \\
		 & =
		f(\mal)L(g) 
		+
		g(\mbe)L(f) 
		-
		g(\mal)L(f) 
		-
		f(\mbe)L(g)   \\
		 & =
		g(\mbe)L(f) 
		-
		g(\mal)L(f)   \\
		 & =
		L(f)(g(\bm{\mbe}) - g(\mal)),
	\end{align*}
	and it follows that $g(\mal) = g(\mbe)$. If we let $L' : A \to \K$ be an
	$\mal, \mbe$-character difference, we've shown that $\Ker(L) \subseteq
		\Ker(L')$. By hypothesis there exists $h \in A$ such that $L'(h) \not = 0$,
	so $L'$ is non-zero whence it follows from Lemma
	\ref{lm-intersect_lin_funcs} that $L = cL'$ for some non-zero $c \in \K$
	and we are done. 
\end{proof}

As $\,_{\mal}\Ka = \Ka$, the previous lemma tells us that we may unify the
definitions of $\mal$-derivations and $\mal, \mbe$-character difference simply
as bimodule derivations from $A$ to $\Kab$ where possibly $\mal = \mbe$. It
turns out that character differences and $\mal$-derivations differ in many
significant ways, and in this paper it will benefit us to consider them as
different kinds of functions most of the time. Nevertheless, it is interesting
to keep the above perspective in mind, and it will be used to shorten proofs
and statements where applicable via the following definition. 
\begin{definition}
	We define an $\mal, \mbe$-subalgebra condition on a $\K$-algebra $A$ to be
	a linear functional which satisfies the Leibniz rule $A \to \Kab$, where
	possibly $\mal = \mbe$.
\end{definition}

\subsection{Subalgebras are Described by $\mal, \mbe$-Subalgebra Conditions}

We will now give result upon which this theory rests. A more general version of
the following result is given in \cite{gorin1969subalgebras}.

\begin{theorem}\label{th-gorin}
    Let $A \subsetneq B \subseteq \Kx$ be algebras such that both $A$ and $B$
    have finite codimension in $\Kx$. Then there exist a subalgebra $C
    \subseteq B$ such that $A = C \cap \Ker(L)$, where $L$ is a $\mal,
    \mbe$-subalgebra condition on $C$ for some $\mal, \mbe \in \K^{n}$.
\end{theorem}

By induction on codimension, it follows that any subalgebra $A \subseteq \Kx$
of finite codimension is the kernel of some finite set of $\mal,
\mbe$-subalgebra conditions, and that given some intermediate subalgebra $A
\subseteq B \subseteq \Kx$, we can choose subalgebra conditions in such a way
that the induced filtration includes $B$. \\ 

When combined with Lemma \ref{lm-intersect_lin_funcs}, we also obtain the
following corollary stating that all subalgebra conditions are $\mal,
\mbe$-subalgebra conditions. This result may also be found in
\cite{Rennison_1970}, where it is obtained via more elementary methods.

\begin{corollary}
    Let $A \subset \Kx$ be a subalgebra of finite codimension, and $L : A \to
    \K$ be a subalgebra condition. Then $L$ is an $\mal, \mbe$-subalgebra
    condition for some $\mal, \mbe \in \K^{n}$.
\end{corollary}
\begin{proof}
    By Theorem \ref{th-gorin}, there is a $\mal, \mbe$-subalgebra condition $L'
    : A \to \K$ such that $\ker(L') = \ker(L)$, after which Lemma
    \ref{lm-intersect_lin_funcs} yields $L = cL'$ for some scalar $c \in \K$. 
    
\end{proof}

\subsection{The Connection Between $\mal$-Derivation Spaces and Cotangent Spaces}

Let $A$ be a $\K$-algebra. It is well known that the space of derivations $A
\to \Ka$ is naturally isomorphic the dual space
$\left(\mathfrak{m}_{\mal}(A)/\mathfrak{m}_{\mal}(A)^2\right)^{*}$. 

\begin{theorem}\label{th-dim_D}
    Let $A \subset \Kx$ be a subalgebra and $\mal \in \K^{n}$. Then
    $\mathcal{D}_{\mal}(A)$ is isomorphic to the dual space
    $\left(\mathfrak{m}_{\mal}(A)/\mathfrak{m}_{\mal}^{2}(A)\right)^{*}$ as
    $A$-modules, and in particular as $\K$-spaces.
\end{theorem}
\begin{proof} 
	See \cite{shafarevich2013basic}, Corollary 2.1 of Chapter 2.1.3,
	along with Exercise 24 of Chapter 2.1.6.
    
\end{proof}

This is a great tool that we will rely heavily upon throughout the paper. As a
first example of its utility, an immediate corollary gives us a bound for
$\dim\left(\mathcal{D}_{\mal}(A)\right)$ by the size of any generating set for
$A$.

\begin{corollary}\label{cor-deriv_bound}
    Let $A \subset \Kx$ be a subalgebra of finite codimension and $G = \{g_{1},
    \ldots, g_m\} \subset A$ be a finite generating set of $A$. Then
    \[
		\dim(\mathcal{D}_{\mal}(A)) \leq m.
	\]
\end{corollary}
\begin{proof}
    By subtracting constants from the elements of $G$, we may suppose that $G
    \subset \mathfrak{m}_{\mal}(A)$. Any element of $h \in
    \mathfrak{m}_{\mal}(A)$ can be written as a polynomial $F(g_{1}, \ldots,
    g_m)$ in the elements of $G$, and as $h(\mal) = g_i(\mal) = 0$, it follows
    that $F$ doesn't have a constant term. Hence $h = F(g_{1}, \ldots, g_m)$ is
    congruent to a linear combination of $g_{1}, \ldots, g_m$ modulo
    $\mathfrak{m}_{\mal}(A)^{2}$ and the corollary follows.
    
\end{proof}

Using Lemma \ref{lm-finit_SAGBI}, we easily obtain the following corollary.
\begin{corollary}
    Let $A \subset \Kx$ be a subalgebra of finite codimension. Then
    $\dim(\mathcal{D}_{\mal}(A))$ is finite.
\end{corollary}

\subsection{Subalgebra Conditions and SAGBI Bases}

If we have a finite minimal SAGBI basis $G$ for a subalgebra $A \subset \Kx$,
and a subalgebra condition $L$ over $A$, then we can construct a finite SAGBI
basis $G'$ of $A \cap \Ker(L)$ as described in the following theorem.

\begin{theorem}\label{th-build_sagbi}
    Given some term order, let $G = \{g_i : i \in [1..m]\}$ be an ordered
    minimal SAGBI basis for $A$, and let $A' = A \cap \Ker(L)$ where
    $L$ is a subalgebra condition over $A$. Let $j$ be the smallest index such
    that $L(g_j) \not = 0$. Then a (not necessarily minimal) SAGBI basis for
    $A'$ is given by 
	\begin{align*}
		G'
		= & 
		\left\{g_i - \frac{L(g_i)}{L(g_j)} g_j : i \not = j
		\right\} \cup                                              \\
		  & \left\{g_ig_j - \frac{L(g_ig_j)}{L(g_j)} g_j : i \in
		[1 .. m]\right\} \cup                                      \\
		  & \left\{g_j^3 - \frac{L(g_j^3)}{L(g_j)} g_j  \right\}.
	\end{align*}
\end{theorem}
\begin{proof}
    First of all, for any $f \in A$ we have $L\left(f -
    \frac{L(f)}{L(g_j)}g_j\right) = 0$, so $G' \subset A'$. \\
	
    It remains to show that $\deg\left(G'\right)$ generates
    $\deg\left(A'\right)$ as a semigroup. First we show that
    $\deg\left(A'\right) = \deg\left(A\right) \setminus \{\deg(g_j)\}$. As $A'$
    has codimension $1$ in $A$, it will suffice to show that $\deg(g_j) \not
    \in \deg\left(A'\right)$. Suppose that $f \in A$ is such that $\deg(f) =
    \deg(g_j)$. After multiplying by a scalar, we can suppose that $f$ has the
    same leading coefficient as $g_j$ whence $\deg(g_j - f) < \deg(g_j)$. Then
    $g_j - f$ can be subduced by elements $g_i \in G$ such that $g_i < g_j$. By
    assumption, all such elements are annihilated by $L$, and thus $L(g_j - f)
    = 0$. It follows that $L(f) = L(g_j) \not = 0$, and $f \not \in A'$. \\	
	
	Now let $\bm{d} \in \deg\left(A'\right) = \deg(A) \setminus
		\{\deg(g_j)\}$. Then as $G$ is a SAGBI basis for $A$ and $A'
		\subset A$, we can write 
	\[
		\bm{d} = \sum_{i = 1}^{m} s_i \deg(g_i)
	\]
	for $s_i \in \N$. Suppose first that $s_j \not = 1$. Then we can write $s_j
		= 2a + 3b$ for $a,b \in \N$, and the sum above may be rewritten as
	\begin{align*}
		\bm{d}
		 & = 
		a \deg\left(g_j^2\right)
		+
		b \deg\left(g_j^3\right)
		+
		\sum_{\substack{i = 1 \\ i \not = j}}^{m} s_i \deg(g_i).
	\end{align*}
	
    Now, if $i < j$ then $L(g_i) = 0$ by assumption, and whenever $i > j$ we
    have a strict inequality $\lm(g_i) > \lm(g_j)$ by the fact that $G$ is
    minimal and ordered. Either way
	\[
		\lm\left(
		g_i - \frac{L(g_i)}{L(g_j)}g_j
		\right)
		=
		\lm(g_i)
	\]
	when $i \not = j$. It is also easy to see that

    \[
        \lm\left(g_ig_j - \frac{L(g_ig_j)}{L(g_j)} g_j\right)
        =
        \lm\left(g_i g_j\right)
    \]
    for all $i \in [1..m]$ and that
    \[
        \lm\left(g_j^3 - \frac{L(g_j^3)}{L(g_j)} g_j\right) = \lm(g_j^{3}).
    \]
    Thus we may write $\bm{d}$ as a natural number combination of elements in
    $\deg\left(G'\right)$ according to
	\begin{align*}
		\bm{d}
		 & = 
		a \deg\left(g_j^2 - \frac{L(g_j^2)}{L(g_j)} g_j\right)
		+
		b \deg\left(g_j^3 - \frac{L(g_j^3)}{L(g_j)} g_j\right)
		+
		\sum_{\substack{i = 1 \\ i \not = j}}^{m} s_i \deg\left(g_i - \frac{L(g_i)}{L(g_j)} g_j\right).
	\end{align*}
	If instead $s_j = 1$, since $\deg(g_j) \not \in \deg\left(A'\right)$, it must
	be the case that $s_r > 0$ for some $r \not = j$. Like above, we can then
	write $\bm{d}$ as a natural number combination of elements in
	$\deg\left(G'\right)$ as 
	\begin{align*}
		\bm{d}
		 & = 
		\deg\left(g_j g_r\right)
		+
		(s_r - 1)\deg\left(g_r\right)
		+
		\sum_{\substack{i = 1 \\ i \not = r, j}}^{m} s_i \deg(g_i) \\
		 & = 
		\deg\left(g_j g_r - \frac{L(g_jg_r)}{L(g_j)}g_j\right)
		+
		(s_r - 1)\deg\left(g_r - \frac{L(g_r)}{L(g_j)}g_j\right)
		+
		\sum_{\substack{i = 1 \\ i \not = r, j}}^{m} s_i \deg\left(g_i - \frac{L(g_i)}{L(g_j)}
		g_j\right).
	\end{align*}
	We have shown that $\deg\left(G'\right)$ generates
	$\deg\left(A'\right)$ as a semigroup, and we are done.
	
\end{proof}

The above theorem may be used to obtain an algorithm which produces a minimal
SAGBI basis for a subalgebra of finite codimension in $\Kx$ which is described
by conditions: First of all, given a finite SAGBI basis $G$ for a subalgebra of
$\Kx$, one can use subduction to determine superfluous elements of $G$ and
obtain a minimal SAGBI basis. Thus given a finite set of subalgebra conditions
$\mathcal{L}$ and a labeling of the conditions $L_i \in \mathcal{L}$ such that
$L_0$ is a subalgebra condition on $A_{0} = \Kx$, and $L_i$ is a subalgebra
condition on $A_i = \ker(L_{i - 1})$, we may obtain a minimal SAGBI basis for
$A = \ker \mathcal{L}$ by iterating through the subalgebra conditions $L_i$ and
using the above theorem along with subduction to obtain minimal SAGBI bases for
each algebra $A_i$ along the filtration. 

\section{Subalgebra Spectrum}

In \cite{gronkvist2022subalgebras}, it was shown that the subalgebra conditions
which define a subalgebra $A$ with finite codimension in the univariate
polynomial algebra $\K[x]$, give rise to a distinguished finite set of points
$\Sp(A) \subset \K$, referred to as the subalgebra spectrum. We now generalize
this notion to the multivariate setting, and obtain many similar results. \\

We need some notation before the definition. When $\bm{u} = (u_1, \ldots, u_n)
	\in \K^n \setminus \{\bm{0}\}$, let $f'_{\bm{u}}$ denote the directional
derivative
\[
	f'_{\bm{u}}
	=
	\sum_{i = 1}^{n} u_i f'_{x_i}.
\]

\begin{definition}
	Let $A \subset \Kx$ be a subalgebra of finite codimension. Then we define
	the subalgebra spectrum of $A$, written $\Sp(A) \subset \K^n$, as the set
    of points $\mal \in \K^{n}$ such that either $f'_{\bm{u}}(\mal) = 0$ for
    all $f \in A$ and some $\bm{u} \in \K^n \setminus \{\bm{0}\}$, or there
    exists some $\mbe \not = \mal$ such that $f(\mal) = f(\mbe)$ for all $f \in
    A$.
\end{definition}

We start by revisiting an example from the introduction.

\begin{example}
	Let 
	\begin{align*}
		A
		=
		\{
		f \in \K[x_1, x_2, x_3]
		:
		 & f'_{x_3}(1, 0, -1) = 0,
		f(3, 2, 5) = f(1, -3, 2),   \\
		 & f'_{x_1}(3,2,5)
		-
		3f'_{x_2}(1, -3, 2) = 0
		\}.
	\end{align*}
    Then $\Sp(A) = \{(1, 0, -1), (3, 2, 5), (1, -3, 2)\}$. In this example, the
    subalgebra spectrum of $A$ consists of all points of evaluation which
    appear in the conditions which define $A$, and as will be shown in the
    upcoming sections, this is the case in general as well.
\end{example}

We will now show that the subalgebra spectrum of a proper subalgebra is
non-empty, but to do this, we first need to classify all derivations on
$\Kx$. A generating set for $\Kx$ inside $\mathfrak{m}_{\mal}$ is given by 
\[
	G = \{x_i - \alpha_i : i \in [1..n]\},
\]
hence $\dim(\mathcal{D}_{\mal}(\Kx)) \leq n$ by Corollary
\ref{cor-deriv_bound}. Now, for each $i \in [1..n]$, we have that $D_i : f
\mapsto f'_{x_i} (\mal)$ is an $\mal$-derivation, and as the $D_i$ are linearly
independent, they span $\mathcal{D}_{\mal}(\Kx)$. We summarize this result in a
lemma.

\begin{lemma}\label{lm-Kx_derivations}
	The $\mal$-derivation space $\mathcal{D}_{\mal}$ of
	$\Kx$ is spanned by the basis $D_i : f \mapsto f'_{x_i}(\mal)$
	for $i \in [1..n]$.
\end{lemma}

We are now ready to prove that the subalgebra spectrum is empty.

\begin{theorem}
	Let $A \subsetneq \Kx$ be a proper subalgebra of finite codimension. Then
	$\Sp(A) \not = \emptyset$.
\end{theorem}
\begin{proof}
    Let $m$ be the codimension of $A$ in $\Kx$. By Theorem \ref{th-gorin}, we
    can find a filtration
    \[
        A = A_m \subsetneq A_{m - 1} \subsetneq \ldots \subsetneq A_{1} \subsetneq A_{0} = \Kx,
    \]
    where each $A_{i + 1}$ is the kernel of some $\mal, \mbe$-subalgebra
    condition $L_{i} : A_i \to \K$. If $L_{0}$ is an $\mal, \mbe$-character
    difference, then $\{\mal, \mbe\} \subset \Sp(A_{1})$. If $L_{0}$ is an
    $\mal$-derivation then $\{\mal\} \subset \Sp(A_{1})$. Either way
    $\Sp(A_{1})$ is non-empty, whence $\Sp(A)$ must be as well.
    
\end{proof}

We conclude this section by defining a class of $\mal$-derivations which we
call trivial.

\begin{definition}
	An $\mal$-derivation over some subalgebra $A \subset \Kx$ of finite
	codimension is said to be trivial when $\mal \not \in \Sp(A)$.
\end{definition}

We call them trivial because, as will be shown in Corollary
\ref{cor-trivial_derivations}, they all exhibit the form given in Lemma
\ref{lm-Kx_derivations}, and therefore $\mathcal{D}_{\mal}(A)$ is only
interesting whenever $\mal \in \Sp(A)$. 

\subsection{Integral Closure, Lifting Morphisms and Isomorphic Subalgebras}

This section demonstrates that our definitions of subalgebra spectrum and
subalgebra conditions behave well under isomorphism. Let $A \subseteq \Kx$ be a
subalgebra of finite codimension. It is well-known that the integral closure
$\overline{A}$ of $A$ in $\Frac(A)$ is equal to $\Kx$, and using this fact it is
easy to show that isomorphic subalgebras of finite codimension in $\Kx$ have
subalgebra spectra and defining conditions which differ only by a polynomial
automorphism. \\

For convenience, we begin by including a proof that $\overline{A} = \Kx$ when
$A$ has finite codimension in $\Kx$.

\begin{lemma}
	Let $A \subset \Kx$ be a subalgebra of finite codimension. Then
	$\overline{A} = \Kx$.
\end{lemma}
\begin{proof}
	We will show that $\Kx \subseteq \overline{A}$ after which equality follows
	by noting that $\Kx$ is a UFD, therefore integrally closed. \\
	
    We begin by showing that $\Kx \subseteq \Frac(A)$. Let $f \in \Kx$. We want
    to show that there exist $g, h \in A$ with $h \not = 0$ such that $f =
    g/h$. This is equivalent to showing that $fA \cap A \not = \{0\}$. Let $m =
    \dim(\Kx/A)$. Then any set of $m + 1$ linearly independent polynomials from
    $fA$ become linearly dependent modulo $A$. Any such non-trivial linear
    combination yields a non-zero element in $fA \cap A$, and it follows that
    $\Kx \subseteq \Frac(A)$. \\
	
    We now show that any polynomial in $\Kx$ is integral over $A$. Let $f \in
    \Kx$ and suppose again that $m = \codim(A)$. Then some $\K$-linear
    combination $\sum_{i = 1}^{m + 1} c_i f^{i}$ must lie in $A$. Let $k$ be
    the greatest index such that $c_k \not = 0$. We then obtain a monic
    polynomial
	\[
        y^{k} + \sum_{i = 1}^{k-1} \frac{c_i}{c_{k}} y^i
		-
        f^{k} - \sum_{i = 1}^{k-1} \frac{c_i}{c_{k}} f^i
	\]
	in $A[y]$ which has $f$ as a root. 
	
\end{proof}

This will be combined with the following well-known result.

\begin{lemma}
    Let $\phi : R \to R'$ be an injective morphism of integral domains. Then
    there is a unique injective morphism $\overline{\phi} : \overline{R} \to
    \overline{R'}$ which restricts to $\phi$, where $\overline{R}$ is the
    integral closure of $R$. Furthermore, if $\phi$ is an isomorphism, then so
    is $\overline{\phi}$.
\end{lemma}

\begin{proof}
    First of all, it is easy to see that there exists a well-defined morphism
    $\widehat{\phi} : \Frac(R) \to \Frac(R')$ given by $\widehat{\phi} : a/b
    \to \phi(a)/\phi(b)$. Moreover, it is easy to verify that $\widehat{\phi}$
    is an isomorphism whenever $\phi$ is. \\
	
	We now let $\overline{\phi}$ be the restriction of $\widehat{\phi}$ to the
	integral closure $\overline{R}$, and show that it lands in $\overline{R'}$.
	Let $a/b \in \Frac(R)$ be an integral element in $R$ which is a root of the
	monic polynomial $f(y) \in R[y]$. Then 
	\[
		0 = \overline{\phi}(0) = \overline{\phi}(f(a/b)) = \phi(f)(\overline{\phi}(a/b))
	\]
    and $\overline{\phi}(a/b)$ is a root of the monic polynomial $\phi(f)(y)
    \in R'[y]$. \\
	
    When $\phi$ is an isomorphism, $\widehat{\phi}$ is as well. We may then
    repeat the argument above for $\widehat{\phi}^{-1}$ and obtain the map
    $\overline{\phi}^{-1} : \overline{R'} \to \overline{R}$. This map is
    necessarily inverse to $\overline{\phi}$, showing that $\overline{\phi}$ is
    an isomorphism. \\
	
    Finally, the extension to the integral closure is unique, for if $a/b \in
    \overline{A}$ we need $\phi(a) = \overline{\phi}(a/b)\phi(b)$ which forces
    $\overline{\phi}(a/b) = \phi(a)/\phi(b)$.

\end{proof}

Now let $A, B \subset \Kx$ be two subalgebras of finite codimension and $\phi :
A \to B$ be a $\K$-algebra isomorphism. From the two previous lemmas it follows
that $\phi$ may be lifted to an automorphism of $\Kx$. Thus if $A$ has
codimension $m$ in $\Kx$, it's immediate that $B = \phi(A)$ has codimension $m$
in $\phi(\Kx) = \Kx$ as well. Moreover, $\phi$ maps each $x_i$ to some $f_i \in
\Kx$ and induces the polynomial automorphism $F : \K^{n} \to \K^{n}$ given by
$F : (\alpha_{1}, \ldots, \alpha_n) \mapsto (f_1(\alpha_{1}), \ldots,
f_n(\alpha_n))$. Then $\phi$ is given by precomposition by $F$, $\phi = F^{*}$,
and for any $f \in \Kx, \mal \in \K^{n}$ we then have that $\phi(f)(\mal) =
F^{*}(f)(\mal) = f(F(\mal))$, and it follows that $\Sp(A) = \{F(\mal) : \mal
\in \Sp(B)\}$. Finally, if
\[
	A = A_{m} \subsetneq A_{m - 1} \subsetneq \ldots \subsetneq A_{0} = \Kx
\]
is a filtration by subalgebras of $\Kx$ landing in $A$ and we let $B_i = \phi(A_i)$, 
it follows that 
\[
	B = B_{m} \subsetneq B_{m - 1} \subsetneq \ldots \subsetneq B_{0} = \Kx
\]
is a filtration by subalgebras of $\Kx$ landing in $B$, and it's easy to verify
that any $\mal, \mbe$-subalgebra condition $L$ on $B_i$ induces a $F(\mal),
	F(\mbe)$-subalgebra condition $L \circ \phi$ on $A_i$. We summarize our results
in a theorem.

\begin{theorem}
	Let $A, B$ be isomorphic subalgebras of finite codimension in $\Kx$. Then
	$A$ and $B$ have the same codimension in $\Kx$, the isomorphism $A \cong B$
	induces a polynomial automorphism $F : \K^{n} \to \K^{n}$ such that $\Sp(A)
		= F(\Sp(B))$, and any set of subalgebra conditions defining $B$ induces a
	set of subalgebra conditions defining $A$ when precomposed with $F^{*}$.
\end{theorem}

\subsection{Clusters}

We now define a central equivalence relation on the subalgebra spectrum. 
\begin{definition}
	Two elements $\mal, \mbe \in \Sp(A)$ are said to be equivalent, written
	$\mal \sim \mbe$ if $f(\mal) = f(\mbe)$ for all $f \in A$. We define the
	clusters of $\Sp(A)$ to be the equivalence classes induced by this relation.
\end{definition}

One immediate justification for the definition is that spectral elements in the
same cluster $\mal \sim \mbe$ share the same derivation space
$\mathcal{D}_{\mal}(A) = \mathcal{D}_{\mbe}(A)$, and we will see in this
section that the connection between clusters and subalgebra conditions runs
deep. \\ 

The following lemma has several important corollaries.

\begin{lemma}\label{lm-sep_conds}
	Let $A \subseteq \Kx$ be a subalgebra of finite codimension, and $L : A \to
		\K$ be a non-zero linear functional which is both an $\mal, \mbe$- and a
	$\mga, \mde$-subalgebra condition. Then either
	\[
		\mal \sim \mga, \mbe \sim \mde,
	\]
	or
	\[
		\mal \sim \mde, \mbe \sim \mga.
	\]
\end{lemma}

\begin{proof}
	The statement of the lemma says that for any $f, g \in A$, we have
	\begin{align}\label{eq-sep_eqs}
		L(fg)
		 & =
		f(\mal)L(g) + g(\mbe)L(f)           \\
		 & =
		f(\mbe)L(g) + g(\mal)L(f) \nonumber \\
		 & =
		f(\mga)L(g) + g(\mde)L(f) \nonumber \\
		 & =
		f(\mde)L(g) + g(\mga)L(f). \nonumber
	\end{align}
	Pick $f \in \mathfrak{m}_{\mal}(A)$ such that $L(f) = 1$, and $g \in A$.
	Then the above equations turn into
	\begin{align*}
		L(fg)
		 & =
		g(\mbe)               \\
		 & =
		f(\mbe)L(g) + g(\mal) \\
		 & =
		f(\mga)L(g) + g(\mde) \\
		 & =
		f(\mde)L(g) + g(\mga),
	\end{align*}
	which we can rearrange to
	\begin{align*}
		f(\mbe)L(g) & = g(\mbe) - g(\mal)   \\
		f(\mga)L(g) & = g(\mbe) - g(\mde)   \\
		f(\mde)L(g) & = g(\mbe) - g(\mga),
	\end{align*}
	and we see that if $\mal \sim \mga$, then $f(\mga) = f(\mal) = 0$, and
	$g(\mbe) - g(\mde) = 0$ for all $g \in A$, i.e $\mbe \sim \mde$. Similarly,
	if $\mal \sim \mde$ then $\mbe \sim \mga$. Moreover, by symmetric
	arguments, it follows that these implications are equivalences. We now just
	need to show that one of these cases occur. \\
	
	Pick $f$ as above and $g \in A$ such that $L(g) = 0$. Then equations
    (\ref{eq-sep_eqs}) turn into
	\[
		g(\mbe) = g(\mal) = g(\mde) = g(\mga).
	\]
	Hence all four points lie in the same cluster in $\Sp(A \cap \Ker(L))$, so
	we only need to check that polynomials outside of $\Ker(L)$ adhere to the
	equalities promised by the lemma. By combining this with our previous
	argument, it follows that we are done if we can show one of:
	\begin{enumerate}
		\item $g(\mal) = g(\mga)$ for all $g \in A \setminus \Ker(L)$,
		\item $g(\mbe) = g(\mde)$ for all $g \in A \setminus \Ker(L)$,
		\item $g(\mal) = g(\mde)$ for all $g \in A \setminus \Ker(L)$,
		\item $g(\mbe) = g(\mga)$ for all $g \in A \setminus \Ker(L)$.
	\end{enumerate}
	Let $g \in \mathfrak{m}_{\mal}(A)$ be such that $L(g) = 1$. Then as
	$\Ker(L)$ has codimension $1$ in $A$, it follows that $A \setminus \Ker(L)
		= \{cg + f : c \in \K, f \in \Ker(L)\}$. As the polynomials in $\Ker(L)$
	satisfy all of the four equalities above, we are done if we can show that
	$g$ satisfies one of them. \\
	
	To do this, we examine two of the ways we can expand $L(g^{3})$ according
	to the various $\mal, \mbe, \mga, \mde$-variations of the Leibniz rule
	available. We have
	\begin{align*}
		L(g^3) 
		 & = g(\mbe)L(g^2) + g^2(\mal)L(g)                           \\
		 & = g(\mbe)g(\mga)L(g) + g(\mbe)g(\mde)L(g) + g^2(\mal)L(g) \\
		 & = g(\mbe)g(\mga) + g(\mbe)g(\mde),                        \\
		L(g^3) 
		 & = g(\mga)L(g^2) + g^2(\mde)L(g)                           \\
		 & = g(\mal)g(\mga)L(g) + g(\mbe)g(\mga)L(g) + g^2(\mde)L(g) \\
		 & = g(\mbe)g(\mga) + g^{2}(\mde),
	\end{align*}
	and by taking the difference of the two expansions we get
	\[
		0
		=
		g(\mde)(g(\mbe) - g(\mde)).
	\]
	Hence either $g(\mbe) = g(\mde)$, or $g(\mde) = 0 = g(\mal)$, and we are
	done.
	
\end{proof}

In particular, if $\mal \sim \mbe$ in $A$ and $L$ is both a $\mal,
	\mbe$-subalgebra condition (I.e an $\mal$-derivation) and a $\mga,
	\mde$-subalgebra condition on $A$, it follows from the lemma that $\mal \sim
	\mbe \sim \mga \sim \mde$ in $A$ and we get the following two corollaries.

\begin{corollary}\label{cor-no-subalg-cond-is-deriv-and-diff}
	No non-zero subalgebra condition is both a character difference and an
	$\mal$-derivation.
\end{corollary}
\begin{corollary}\label{cor-deriv_disjoint}
	Let $A \subset \Kx$ be a subalgebra of finite codimension where $\mal \not
		\sim \mga$. Then $\mathcal{D}_{\mal}(A) \cap \mathcal{D}_{\mga}(A) = 0$.
\end{corollary}

The following corollary is also an immediate consequence of the lemma.

\begin{corollary}\label{cor-merge_two_clusters}
	Two character differences agree if and only if they pertain to the same two
	clusters. Thus kerneling by a non-zero character difference merges
	exactly two clusters.
\end{corollary}

The remainder of this section will be dedicated to determining how the
$\mal$-derivation space of a subalgebra $A \subseteq \Kx$ changes when we
kernel by subalgebra conditions which aren't $\mal$-derivations. We are going
to need a slightly more flexible notation for juggling derivation spaces of
multiple subalgebras simultaneously. If $A' \subseteq A$ is a subalgebra, we
write $\restr{\mathcal{D}_{\mal}(A)}{A'}$ for the space of functions
$\mathcal{D}_{\mal}(A)$ restricted to elements in $A'$.

\begin{lemma}\label{lm-dim_deriv_constant_kernel_outside_cluster}
	Let $A \subset \Kx$ be a subalgebra of finite codimension, $L$ a subalgebra
	condition over $A$, and $A' = A \cap \Ker(L)$. If $L$ is not an
	$\mal$-derivation, then 
	\[
		\dim\left(\restr{\mathcal{D}_{\mal}(A)}{A'}\right)
		=
		\dim(\mathcal{D}_{\mal}(A))
	\]
\end{lemma}
\begin{proof}
	Let $D_1, D_2, \ldots D_N \in \mathcal{D}_{\mal}(A)$
	be a vector space basis for $\mathcal{D}_{\mal}(A)$. Assume 
	towards a contradiction that the $D_i$ admit a non-trivial linear 
	dependency when restricted to $A'$,
	\[
		0 = \sum_{i = 1}^N a_i \restr{D_i}{A'}.
	\]
	Then by Lemma \ref{lm-intersect_lin_funcs}, we have 
	\[
		\sum_{i = 1}^N a_i D_i = cL
	\]
	for some non-zero scalar $c \in \K$. This is a contradiction as the
	expression on the left is a non-zero $\mal$-derivation, and the expression
	on the right is not.
	
\end{proof}

We can use the previous lemma along with Theorem \ref{th-dim_D} to prove a
lemma which will be one of our main technical tools for the remainder of this
section. 

\begin{lemma}\label{lm-look_at_M2}
	Let $A \subset \Kx$ be a subalgebra of finite codimension, and $L$ a
	subalgebra condition over $A$ which is not an $\mal$-derivation. Denote $A'
		= A \cap \Ker(L)$. Then 
	\[
		\dim\left(
		\mathcal{D}_{\mal}(A')
		/
		\restr{\mathcal{D}_{\mal}(A)}{A'}
		\right)
		=
		\dim\left( \mathfrak{m}_{\mal}(A)^2
		/
		\mathfrak{m}_{\mal}(A')^2
		\right) - 1.
	\]
\end{lemma}
\begin{proof}
	Lemma \ref{lm-dim_deriv_constant_kernel_outside_cluster} and Theorem
	\ref{th-dim_D} yield
	\begin{align*}
		\dim\left(
		\mathcal{D}_{\mal}(A')
		/
		\restr{\mathcal{D}_{\mal}(A)}{A'}
		\right)
		= & 
		\dim(\mathcal{D}_{\mal}(A'))
		-
		\dim(\restr{\mathcal{D}_{\mal}(A)}{A'}) \\
		= & 
		\dim(\mathcal{D}_{\mal}(A'))
		-
		\dim(\mathcal{D}_{\mal}(A))             \\
		= & 
		\dim\left(
		\mathfrak{m}_{\mal}(A')/ \mathfrak{m}_{\mal}(A')^2
		\right)
		-
		\dim\left(
		\mathfrak{m}_{\mal}(A)/ \mathfrak{m}_{\mal}(A)^2
		\right)                                 \\
		= & 
		\codim\left(\mathfrak{m}_{\mal}(A')^2\right)
		-
		\codim(\mathfrak{m}_{\mal}(A'))         \\
		  & -
		\codim\left(\mathfrak{m}_{\mal}(A)^2\right)
		+
		\codim(\mathfrak{m}_{\mal}(A))          \\
		= & 
		\codim\left(\mathfrak{m}_{\mal}(A')^2\right)
		-
		\codim\left(\mathfrak{m}_{\mal}(A)^2\right)
		- 1                                     \\
		= & 
		\dim\left(\mathfrak{m}_{\mal}(A)^2 / \mathfrak{m}_{\mal}(A')^2\right)
		- 1.
	\end{align*}
	where we take codimensions in $\Kx$.
	
\end{proof}

We can specify the previous lemma further. Consider two polynomials $f_1, f_2
	\in \mathfrak{m}_{\mal}(A)$ and let $g$ be a polynomial in
$\mathfrak{m}_{\mal}(A) \setminus \mathfrak{m}_{\mal}(A')$. Then
$\mathfrak{m}_{\mal}(A) = \mathfrak{m}_{\mal}(A') \oplus g \K$ since
$\codim(\mathfrak{m}_{\mal}(A)) = \codim(\mathfrak{m}_{\mal}(A')) + 1$, and we
can write $f_i = h_i + a_i g$ for some $h_i \in \mathfrak{m}_{\mal}(A')$ and
$a_i \in \K$. Moreover, 
\[
	f_1f_2
	=
	h_1h_2 + (a_2h_1 + a_1h_2 + a_1a_2 g) g,
\]
and any polynomial in $\mathfrak{m}_{\mal}(A)^2$ is a linear combination of such products.
Thus any polynomial in $\mathfrak{m}_{\mal}(A)^2$ is congruent to a product $f g$ for some
$f \in \mathfrak{m}_{\mal}(A)$ when taken modulo $\mathfrak{m}_{\mal}(A')^2$. We summarize our
result.

\begin{lemma}\label{lm-look_at_gM}
	Using notation as in Lemma \ref{lm-look_at_M2} and letting $g \in
		\mathfrak{m}_{\mal}(A) \setminus \mathfrak{m}_{\mal}(A')$ we have
	\[
		\dim\left(
		\mathcal{D}_{\mal}(A')
		/
		\restr{\mathcal{D}_{\mal}(A)}{A'}
		\right)
		=
		\dim\left(
		\left(
			g \mathfrak{m}_{\mal}(A) + \mathfrak{m}_{\mal}(A')^2 
			\right)
		/
		\mathfrak{m}_{\mal}(A')^2 
		\right) - 1.
	\]
\end{lemma}

It will often be more convenient to show dimensional equalities by counting
leading terms, and we will often use the previous lemma in conjunction with the
following adaptation of Macaulay's Basis Theorem.

\begin{lemma}\label{lm-eq_dim_lm}
	Let $B \subseteq A \subseteq \Kx$ be vector spaces of polynomials such that
	$\dim(A/B)$ is finite. Then 
	\[
		\dim(A/B) = |\Lm(A) \setminus \Lm(B)|
	\]
\end{lemma}
\begin{proof}
    Our proof will closely follow that of Macaulay's Basis Theorem given in
    \cite[Theorem 1.5.7]{kreuzerrobbiano1}. \\

    Let $F$ be a subset of $A$ such that $\Lm(F) = \Lm(A) \setminus \Lm(B)$,
    and such that all leading monomials of elements of $F$ are distinct. Then
    \[
        \vert F \vert = \vert \Lm(F) \vert = \vert \Lm(A) \setminus \Lm(B) \vert,
    \]
    and it remains to show that $\vert F \vert = \dim(A/B)$. We will do so by
    showing that the residue classes of the elements of $F$ form a basis of
    $A/B$. \\ 

    We first show linear independence. Since the elements of $F$ have distinct
    leading monomials, they are linearly independent. Moreover, as $\Lm(F)$ is
    disjoint from $\Lm(B)$, any non-trivial linear combination of elements of
    $F$ cannot lie in $B$. Thus the elements of $F$ remain linearly independent
    modulo $B$. \\

    Now let $a_{0} \in A$. Them $\lm(a_{0})$ lies in either $\Lm(F)$ or
    $\Lm(B)$. In either case, there exists a polynomial $f_{0} \in \langle F
    \rangle + B$ such that $\lm(a_{0} + f_{0}) < \lm(f_{0})$. Now let $a_{1} =
    a_{0} + f_{0}$. As before, we may find $f_{1} \in \langle F \rangle + B$ such
    that $a_{2} = a_{1} + f_{1}$ has smaller lead monomial than $a_{1}$.
    Repeating this procedure, we obtain a decreasing sequence of lead monomials
    $\lm(a_i), i \in \N$. As any term order is a well-order, such a sequence
    must stabilize, and there exists $N \in \N$ such that $a_N = 0$. We have
    obtained $a_{0} = f_{0} + \ldots + f_N \in \langle F \rangle + B$, and as
    $a_{0} \in A$ was arbitrary, we have $\langle F \rangle + B = A$. Hence the
    residue classes of the elements in $F$ span the quotient space $A/B$ and we
    are done. 
\end{proof}

We are now ready to prove the first difficult statement regarding how
$\mal$-derivation spaces change as we apply subalgebra conditions. 

\begin{theorem}\label{th-conditions_on_other_clusters}
	Let $A \subset \Kx$ be a subalgebra of finite codimension where
	$\mga \not \sim \mal, \mbe$. Let $L$ be an $\mal,
		\mbe$-subalgebra condition,
	and $A' = A \cap \Ker(L)$. Then 
	\[
		\mathcal{D}_{\mga}(A')
		=
		\restr{\mathcal{D}_{\mga}(A)}{A'}
	\]
\end{theorem}
\begin{proof}
    Let $g$ be an element in $\mathfrak{m}_{\mga}(A) \setminus
    \mathfrak{m}_{\mga}(A')$ of minimal degree. Furthermore, after dividing $g$
    by $L(g)$ we may suppose that $L(g) = 1$. By Lemmas \ref{lm-look_at_gM} and
    \ref{lm-eq_dim_lm}, we will be done if we can show that 
	\[
		\left|
		\Lm(g \mathfrak{m}_{\mga}(A) + \mathfrak{m}_{\mga}(A')^2)
		\setminus
		\Lm(\mathfrak{m}_{\mga}(A')^2)
		\right|
		\leq
		1,
	\]
	or equivalently that
	\[
		\left|
		\Lm(g \mathfrak{m}_{\mga}(A))
		\setminus
		\Lm(\mathfrak{m}_{\mga}(A')^2)
		\right|
		\leq
		1.
	\]
	Thus, our objective will be to show that all but one of the leading
	monomials which appear in $g \mathfrak{m}_{\mga}(A)$ also appear in
	$\mathfrak{m}_{\mga}(A')^2$. \\
	
	Let $\hat{B}$ be a vector space basis for $\mathfrak{m}_{\mga}(A)$ that
	contains $g$ where no polynomials in $\hat{B}$ have the same leading
	monomial. Note that if $f \in \hat{B}$, then $f - L(f)g \in
		\mathfrak{m}_{\mga}(A')$ and $\lm(f - L(f)g) = \lm(f)$ since if $\lm(g) >
		\lm(f)$ we have $L(f) = 0$ by our choice of $g$. Use $\hat{B}$ to construct 
	\[
		B = \{f - L(f)g : f \in \hat{B} \setminus \{g\}\} \cup \{g\}.
	\]
	Then $B$ is a vector space basis for $\mathfrak{m}_{\mga}(A)$ where all
	elements have distinct leading terms and $B \setminus \{g\}$ is a vector
	space basis for $\mathfrak{m}_{\mga}(A')$. \\
	
	We claim that there exists $h \in B \setminus \{g\}$ such that $h(\mbe)
		\not = 0$. To see this, first note that $L$ is not a $\mga, \mbe$-character
	difference by Corollary \ref{cor-no-subalg-cond-is-deriv-and-diff}. It then
	follows from Lemma \ref{lm-intersect_lin_funcs} that $\mga \not \sim \mbe$
	in $A'$. Then as $B \setminus \{g\}$ is a vector space basis for
	$\mathfrak{m}_{\mga}(A')$, we must be able to find such $h \in B \setminus
		\{g\}$ such that $h(\mbe) \not = 0$. \\
	
	Now let $h$ be the minimal polynomial in $B \setminus \{g\}$ with respect
	to the given term order such that $h(\mbe) \not = 0$, and let $f$ be an
	arbitrary element in $B \setminus \{g, h\}$. Then we have the following
	inclusions
	\begin{align*}
		h & \in \mathfrak{m}_{\mga}(A'),  \\
		f & \in \mathfrak{m}_{\mga}(A'),
	\end{align*}
	and,
	\begin{align*}
		fg - \frac{L(fg)}{L(g)}g 
		= fg - \frac{L(f)g(\mal) + L(g)f(\mbe)}{L(g)}g
		= fg - f(\mbe) g & \in \mathfrak{m}_{\mga}(A'),  \\	
		hg - \frac{L(hg)}{L(g)}g 
		= hg - \frac{L(h)g(\mal) + L(g)h(\mbe)}{L(g)}g
		= hg - h(\mbe) g & \in \mathfrak{m}_{\mga}(A').
	\end{align*}
	Using these inclusions we see that 
	\begin{align*}
		\left(
		fg - f(\mbe)g
		\right)h
		 & =
		fgh - f(\mbe)gh
		\in 
		\mathfrak{m}_{\mga}(A')^2, \\	
		\left(
		hg - h(\mbe)g
		\right)f
		 & =
		fgh - h(\mbe)gf
		\in 
		\mathfrak{m}_{\mga}(A')^2,
	\end{align*}
	but then it follows that 
	\[
		f(\mbe)gh - h(\mbe)gf
		\in 
		\mathfrak{m}_{\mga}(A')^2 \cap g \mathfrak{m}_{\mga}(A),
	\]
    and $\lm(f(\mbe)hg - h(\mbe)fg) = \lm(fg)$ since either $f > h \Rightarrow
    fg > hg$ or $f < h$ which means that $f(\mbe) = 0$ due to how we picked
    $h$. Thus, we have shown that $\Lm(\mathfrak{m}_{\mga}(A')^2 \cap g
    \mathfrak{m}_{\mga}(A))$ contains all such elements $\lm(fg)$, i.e all
    elements of $\Lm(g \mathfrak{m}_{\mga}(A))$ except possibly either of
    $\lm(gh), \lm(g^2)$. We can find the last missing lead monomial in a manner
    similar to that above by noting that $g^2 - (g(\mal) + g(\mbe))g \in
    \mathfrak{m}_{\mga}(A')$ and 
	\begin{align*}
		h^2(g^2 - (g(\mal) + g(\mbe))g)
		 & =
		h^2g^2 - (g(\mal) + g(\mbe))h^2g 
		\in \mathfrak{m}_{\mga}(A')^2 \\
		(hg - h(\mbe)g)^2
		 & =
		h^2g^2 - 2h(\mbe) hg^2 + h^2(\mbe)g^2
		\in \mathfrak{m}_{\mga}(A')^2,
	\end{align*}
	so 
	\begin{align*}
		2h(\mbe) hg^2 
		- (g(\mal) + g(\mbe))h^2g 
		- h^2(\mbe)g^2 \in \mathfrak{m}_{\mga}(A')^2.
	\end{align*}
	We also have 
	\begin{align*}
		h(hg - h(\mbe)g) 
		 & =
		h^2g - h(\mbe)hg
		\in \mathfrak{m}_{\mga}(A')^2, \\	
		h(g^2 - (g(\mal) + g(\mbe))g) 
		 & =
		hg^2 - (g(\mal) + g(\mbe))hg
		\in \mathfrak{m}_{\mga}(A')^2, 
	\end{align*}
	and combing the results we see that 
	\begin{align*}
		2h(\mbe)(g(\mal) + g(\mbe))hg 
		-
		h(\mbe)(g(\mal) + g(\mbe))hg 
		-
		h^2(\mbe)g^2 
		 & =                                                           \\
		h(\mbe)(g(\mal) + g(\mbe))hg 
		-
		h^2(\mbe)g^2 
		 & \in \mathfrak{m}_{\mga}(A')^2 \cap g \mathfrak{m}_{\mga}(A)
	\end{align*}
	Whether or not $g(\mal) + g(\mbe) = 0$ does not matter, the leading monomial
	of the polynomial above is either $\lm(hg)$ or $\lm(g^2)$, whence we see
	that
	\[
		\dim\left(
		\Lm(g \mathfrak{m}_{\mga}(A))
		\setminus
		\Lm(\mathfrak{m}_{\mga}(A')^2 \cap g\mathfrak{m}_{\mga}(A))
		\right)
		\leq
		1
	\]
	and we are done.
	
\end{proof}

If we combine the previous theorem with Lemma \ref{lm-Kx_derivations}, we
immediately get the following promised result regarding trivial derivations.

\begin{corollary}\label{cor-trivial_derivations}
	Let $A \subset \K[x]$ be a subalgebra of finite codimension 
	such that $\mal \not \in \Sp(A)$. Then $\mathcal{D}_{\mal}$
	is spanned by the $\mal$-derivations 
	\[
		D_i(f) = f'_{x_i}(\mal)
	\]
	for $x_i \in \bm{x}$.
\end{corollary}

With the previous corollary in hand, we can now further solidify the idea that
the subalgebra spectrum consists of all points of evaluation that appear in the
subalgebra conditions which define $A$. To fully resolve this, we need the Main
Theorem (Theorem \ref{th-main}). What we can show now however is that $\mal \in
	\Sp(A)$ if and only if some $\mal$-derivation, or $\mal, \mbe$-character
difference is among the subalgebra conditions which define $A$.

\begin{corollary}
    Let $B \subset \Kx$ be a subalgebra of finite codimension, $L$ be a
    non-zero $\mal, \mbe$-subalgebra condition on $B$, and $A = \ker(L)$. Then 
	\[
		\Sp(A) = \{ \mal, \mbe \} \cup \Sp(B).   
	\]
\end{corollary}
\begin{proof}
    It is immediate from the definitions and Corollary
    \ref{cor-trivial_derivations} that $\{ \mal, \mbe \} \cup \Sp(B) \subseteq
    \Sp(A)$. We need to show that no unexpected elements appear in $\Sp(A)$.
    Suppose that $\mga \in \Sp(A) \setminus \Sp(B)$. Then by the definition of
    the subalgebra spectrum, we have $L'(f) = 0$ for all $f \in A$ and $L'(g)
    \not = 0$ for some $g \in B$ where $L'$ is either a $\mga, \mde$-character
    difference for some $\mde$, or $L'$ is some $\mga$-derivation given by
    $L'(f) = f'_{\bm{u}}(\mga)$ for some $\bm{u} \in \K^{n} \setminus
    \{\bm{0}\}$. As both $L, L'$ are non-zero on $B$, but zero on $A$ which has
    codimension $1$ in $B$, it follows from Lemma \ref{lm-intersect_lin_funcs}
    that $L = cL'$ for some $c \in \K$. By Lemma \ref{lm-sep_conds}, we have
    $\mal \sim \mga$ or $\mbe \sim \mga$ in $B$. Suppose without loss of
    generality that $\mal \sim \mga$ in $B$. Then $\mal = \mga$, as otherwise
    we would have $h(\mal) = h(\mga)$ for all $h \in B$, and the definition of
    the subalgebra spectrum tells would give the contradictory statement $\mga
    \in \Sp(B)$.
	
\end{proof}

Our remaining objective for this section is to show that kerneling by $\mal,
\mbe$-character differences merges the  $\mal$- and $\mbe$-derivations spaces
in a direct sum. We will do this in two parts, the first of which is the
following lemma.

\begin{lemma}\label{lm-derivclusters_disjoint}
    Let $A \subset \Kx$ be a subalgebra of finite codimension where $\mal \not
    \sim \mbe$. Let $L$ be an $\mal, \mbe$-character difference and denote $A'
    = A \cap \Ker(L)$. Then
    \[
        \restr{\mathcal{D}_{\mal}(A)}{A'} \cap
        \restr{\mathcal{D}_{\mbe}(A)}{A'} = 0
    \]
    and
	\[
		\restr{\mathcal{D}_{\mal}(A)}{A'}
		\oplus
		\restr{\mathcal{D}_{\mbe}(A)}{A'}
		\subseteq 
		\mathcal{D}_{\mal}(A').
	\]
\end{lemma}
\begin{proof}
    Let $D_1 \not = 0 \in \mathcal{D}_{\mal}(A)$ and $D_{2} \not = 0 \in
    \mathcal{D}_{\mbe}(A)$. We are done if we can show that $\restr{D_1}{A'}
    \not = \restr{D_2}{A'}$. \\
        
    Assume towards a contradiction that $\restr{D_1}{A'} = \restr{D_2}{A'}$. As
	$\mal \not \sim \mbe$ in $A$, we can find $f \in A$ such that $f(\mal) = 1$
	and $f(\mbe) = 0$. Then $f \not \in A'$, and as $A'$ has codimension $1$ in
	$A$, we have $A = A' \oplus f\K$. Now let $h_k = f^k - f$ for $k > 1$ and
	note that $h_k \in A'$. As $D_1$ and $D_2$ coincide here, we get 
	\begin{align*}
		0
		 & =
		D_1(h_k) - D_2(h_k)         \\
		 & =
		D_1(f^k - f) - D_2(f^k - f) \\
		 & =
		k 1^{k-1} D_1(f) - D_1(f)
		-
		k 0^{k-1} D_2(f)
		+ 
		D_2(f)                      \\
		 & =
		(k-1)D_1(f) + D_2(f).
	\end{align*}
    As this needs to hold for all $k > 1$, we see that $D_1(f) = D_2(f) = 0$.
    Now, $D_1 = D_2$ on $A'$ by assumption, and as $A = A' \oplus f\K$, we see
    that $D_1 = D_2$ on all of $A$, contradicting Corollary
    \ref{cor-deriv_disjoint}.
	
\end{proof}

\begin{theorem}\label{th-connect_clusters_direct_sum_deriv_space}
	Let $A \subset \Kx$ be a subalgebra of finite codimension where
	$\mal \not \sim \mbe$. Let $L$ be an $\mal,
		\mbe$-character difference and denote $A' = A \cap \Ker(L)$. Then 
	\[
		\mathcal{D}_{\mal}(A')
		=
		\restr{\mathcal{D}_{\mal}(A)}{A'}
		\oplus
		\restr{\mathcal{D}_{\mbe}(A)}{A'}
	\]
\end{theorem}
\begin{proof}
	Lemma \ref{lm-derivclusters_disjoint} tells us that 
	\[
		\restr{\mathcal{D}_{\mal}(A)}{A'}
		\oplus
		\restr{\mathcal{D}_{\mbe}(A)}{A'}
		\subseteq
		\mathcal{D}_{\mal}(A'),
	\]
	so we are done if we can show that 
    \begin{equation}\label{ineq_deriv-disjoin-main-ineq}
		\dim\left(\mathcal{D}_{\mal}(A')/\restr{\mathcal{D}_{\mal}(A)}{A'}\right)
		\leq
		\dim\left(\restr{\mathcal{D}_{\mbe}(A)}{A'}\right).
	\end{equation}
    We now wish to translate this using Lemma \ref{lm-look_at_gM} and Theorem
    \ref{th-dim_D}. Thus, we let $g$ be an element in $\mathfrak{m}_{\mal}(A)
    \setminus \mathfrak{m}_{\mal}(A')$ of minimal degree. Then $g(\mbe) \not =
    0$, and we may suppose that $g(\mbe) = -1$. Showing that inequality
    (\ref{ineq_deriv-disjoin-main-ineq}) holds is then equivalent to showing that
	\begin{equation*}
		\dim\left(
		\left(
			g \mathfrak{m}_{\mal}(A) + \mathfrak{m}_{\mal}(A')^2 
			\right)
		/
		\mathfrak{m}_{\mal}(A')^2
		\right)
		-
		1
		\leq
		\dim\left(
		\mathfrak{m}_{\mbe}(A)
		/
		\mathfrak{m}_{\mbe}(A)^2 
		\right).
	\end{equation*}
	Now, as $L$ is a linear functional, we either have that
	\[
		\dim\left(
		\left(
			g \mathfrak{m}_{\mal}(A) + \mathfrak{m}_{\mal}(A')^2 
			\right)
		/
		\mathfrak{m}_{\mal}(A')^2
		\right) - 1
		=
		\dim\left(
		\left(
			g \mathfrak{m}_{\mal}(A') + \mathfrak{m}_{\mal}(A')^2 
			\right)
		/
		\mathfrak{m}_{\mal}(A')^2
		\right),
	\]
	or that
	\[
		\dim\left(
		\left(
			g \mathfrak{m}_{\mal}(A) + \mathfrak{m}_{\mal}(A')^2 
			\right)
		/
		\mathfrak{m}_{\mal}(A')^2
		\right)
		=
		\dim\left(
		\left(
			g \mathfrak{m}_{\mal}(A') + \mathfrak{m}_{\mal}(A')^2 
			\right)
		/
		\mathfrak{m}_{\mal}(A')^2
		\right).
	\]
	Either way, 
	\[
		\dim\left(
		\left(
			g \mathfrak{m}_{\mal}(A) + \mathfrak{m}_{\mal}(A')^2 
			\right)
		/
		\mathfrak{m}_{\mal}(A')^2
		\right) - 1
		\leq	
		\dim\left(
		\left(
			g \mathfrak{m}_{\mal}(A') + \mathfrak{m}_{\mal}(A')^2 
			\right)
		/
		\mathfrak{m}_{\mal}(A')^2
		\right),
	\]
	and it will be enough to show that
	\begin{equation}\label{eq-main_id}
		\dim\left(
		\left(
			g \mathfrak{m}_{\mal}(A') + \mathfrak{m}_{\mal}(A')^2 
			\right)
		/
		\mathfrak{m}_{\mal}(A')^2
		\right)
		\leq
		\dim\left(
		\mathfrak{m}_{\mbe}(A)
		/
		\mathfrak{m}_{\mbe}(A)^2 
		\right),
	\end{equation}
	which will be our main objective for the remainder of this proof. \\
	
	Before we start, note that $\mathfrak{m}_{\mal}(A') = \mathfrak{m}_{\mbe}(A')$
	since $\mal \sim \mbe$ in $A'$ (but not in $A$ of
	course). These different ways of writing the same space will be
	used interchangeably throughout the proof. \\
	
	Like in the proof of Theorem \ref{th-conditions_on_other_clusters}, let $B$
	be a vector space basis for $\mathfrak{m}_{\mal}(A)$ containing $g$, where
	all polynomials of $B$ have distinct leading monomials, and $B \setminus
		\{g\}$ is a vector space basis for $\mathfrak{m}_{\mal}(A')$. \\
	
	We construct another basis very similar to $B$. Let $\hat{g} = g -
		g(\mbe)$, and $\widehat{B} = B \cup \{\widehat{g}\} \setminus \{g\}$. Then
	$\widehat{g}(\mal) = 1, \widehat{g}(\mbe) = 0$ and $\widehat{B}$ is a
	vector space basis for $\mathfrak{m}_{\mbe}(A)$. \\
	
	We will prove inequality (\ref{eq-main_id}) by showing that, given any $m$
	linearly independent elements in $\left(g \mathfrak{m}_{\mal}(A') +
		\mathfrak{m}_{\mal}(A')^2\right) / \mathfrak{m}_{\mal}(A')^2$, we can
	always find $m$ linearly independent elements in $\mathfrak{m}_{\mbe}(A) /
		\mathfrak{m}_{\mbe}(A)^2$. More precisely, let $F \subset
		\mathfrak{m}_{\mal}(A')$ be a finite set of polynomials such that $gF$ is a
	linearly independent set modulo $\mathfrak{m}_{\mal}(A')^2$. We will show
	that the elements of $F$ are linearly independent modulo
	$\mathfrak{m}_{\mbe}(A)^2$. \\
	
	We consider a linear relation among $F$ modulo $\mathfrak{m}_{\mbe}(A)^2$.
	Assume towards a contradiction that there exists a set of scalars $\{a_i \in
		\K : f_i \in F\}$ such that some $a_i \not = 0$ and 
	\[
		\sum_{f_i \in F} a_i f_i 
		= 
		\sum_{(h_i, H_i) \in H} b_i h_i H_i 
	\]
	for some finite $H \subset \widehat{B} \times \widehat{B}$, and scalars
	$b_i \in \K$. Denote $p = \sum_{(h_i, H_i) \in H} b_i h_i H_i$. Then 
	\[
		g\sum_{f_i \in F} a_i f_i = g p,
	\]
	and it follows that $gp \not \in \mathfrak{m}_{\mal}^{2}(A')$ since $gF$ is
	linearly independent modulo $\mathfrak{m}_{\mal}(A')^2$. I.e, we have $p$ such
	that $p \in \mathfrak{m}_{\mbe}(A)^2$ but $gp \not \in \mathfrak{m}_{\mbe}(A')^2$. \\
	
	Now we will show that $gh_iH_i \in \mathfrak{m}^{2}_{\mbe}(A')$ for each
	$(h_i, H_i) \in H \setminus (\widehat{g}, \widehat{g})$. We have that at
	least one of $h_i$ or $H_i$ is not equal to $\widehat{g}$, say $h_i$,
	whence $h_i \in \mathfrak{m}_{\mbe}(A')$. Moreover, as $H_i \in \hat{B}$ we
	have $H_i(\mbe) = 0$, and as $g(\mal) = 0$, we have 
	\[
		L(gH_i)
		=
		g(\mal)H_i(\mal)
		-
		g(\mbe)H_i(\mbe)
		=
		0,
	\]
	so $gH_i \in \mathfrak{m}_{\mbe}(A')$, and
	\[
		gh_i H_i
		\in 
		\mathfrak{m}_{\mbe}(A')^2 
		\text{ for }
		(h_i, H_i)
		\in 
		H \setminus (\widehat{g}, \widehat{g}).
	\]
	Let $s = 1$ if $(\widehat{g}, \widehat{g}) \in H$ and $s = 0$ otherwise.
	Then
	\begin{align*}
		gp
		 & =
		\sum_{(h_i, H_i) \in H} b_i gh_i H_i \\
		 & =
		sc g \widehat{g}^2 
		+
		\sum_{(h_i, H_i) \in H \setminus (\widehat{g}, \widehat{g})} b_i gh_i H_i
	\end{align*}
	for some scalar $c \not = 0\in \K$. We see that $gp \in
		\mathfrak{m}_{\mbe}(A')^2$ if $s = 0$ so $s$ must be $1$.  Now note that the sum
	\[
		\sum_{(h_i, H_i) \in H \setminus (\widehat{g}, \widehat{g})}b_i g h_i H_i
	\]
	lies in $g\mathfrak{m}_{\mal}(A)$, since for each $h_i, H_i$, either $h_i(\mal) =0$ or
	$H_i(\mal) = 0$. Moreover, $g \widehat{g}^2 \not \in g\mathfrak{m}_{\mal}(A)$ since
	$E(\widehat{g}^2) = \widehat{g}^2(\mal)- \widehat{g}^2(\mbe) = -1$, so the
	sum
	\[
		cg \widehat{g}^2 
		+
		\sum_{(h_i, H_i) \in H \setminus (\widehat{g}, \widehat{g})} b_i g h_i H_i
	\]
	does not lie in $g \mathfrak{m}_{\mal}(A)$, which is a contradiction to $gF \in
		g \mathfrak{m}_{\mal}(A)$ and we are done.
	
\end{proof}

\section{$\mal$-Derivations as Derivative Evaluations.}

In the univariate case we know from
\cite{gronkvist2022subalgebras,leffler2021kandidat} that we can write
$\alpha$-derivations in $A \subseteq \K[x]$ as linear combinations of
derivative evaluations in the spectral elements which are equivalent to
$\alpha$ in $A$. Moreover, we saw in Corollary \ref{cor-trivial_derivations}
that trivial derivations are given as linear combinations of evaluations
of partial derivatives. The purpose of this section is to explore how these
statements generalize to subalgebras $A \subseteq \Kx$ of finite codimension.
\\

We will begin by introducing some notation. We then take a slight detour to
build some intuition by determining $\mathcal{D}_{\mal}(A)$ for subalgebras $A
\subset \Kx$ of codimension $1$. We then state and prove a generalization of
the Main Theorem given in \cite{gronkvist2022subalgebras,leffler2021kandidat}
to subalgebras $A \subset \Kx$ of finite codimension in the multivariate case.

\subsection{Notation and General Leibniz Rule for Directional Derivatives}

When we have a multiset $U = [\bm{u}_1, \bm{u}_2, \ldots, \bm{u}_k]$ with each
$\bm{u}_i \in \K^{n}$, and we want to compose the corresponding directional
derivatives one after another, we write $f^{(k)}_{U}$. If $k$ is small, we may
write $f^{(k)}_{\bm{u}_1 \bm{u}_2 \ldots \bm{u}_k}$. We use multisets as
directional derivatives commute and may be applied with multiplicity greater
than $1$. \\

We will often abuse notation and let $x_i$, whenever it appears in some
multiset describing a higher order directional derivative, denote the
element $(0, \ldots, 0, 1, 0, \ldots, 0) \in \K^{n}$ where the $i$-th index is
$1$, and the remaining indices are $0$. This allows us to write expressions
like $f'_{x_i}$, $f^{(3)}_{x_{1}x_{2}x_{1}}$, and $f^{(k)}_{x_{1} \bm{u}_2
\ldots \bm{u}_k}$. \\

The General Leibniz Rule is given by 
\begin{align*}
	(fg)^{(k)}_{U}
	 & =
	\sum_{U' \in \mathcal{P}(U)}
	f^{(|U'|)}_{U'}
	g^{(|U| - |U'|)}_{U \setminus U'}
\end{align*}
where $\mathcal{P}(U)$ is the power-multiset of the multiset $U$. Note that
$\mathcal{P}(U)$ will contain duplicates if $U$ does. \\ 

We introduce one more notation. We write $\bm{d}^j$ to be the set of all
multisets of combinations of $j$ elements from $\bm{x} = \{x_1, x_2, \ldots
x_n\}$ drawn with repetitions. Essentially, we use $\bm{d}^j$ to represent all
possible combinations of $j$ pure (as in not directional) higher order partial
derivatives. I.e if $n = 3$ then 
\begin{align*}
	\bm{d}^3 
	= & 
	\{
	[x_1, x_1, x_1],
	[x_1, x_1, x_2],
	[x_1, x_1, x_3],
	[x_1, x_2, x_2],
	[x_1, x_2, x_3],      \\
	  & [x_1, x_3, x_3],
	[x_2, x_2, x_2],
	[x_2, x_2, x_3],
	[x_2, x_3, x_3],
	[x_3, x_3, x_3]
	\}
\end{align*}
Each $\bm{d}^j$ will be of size 
\[
	|\bm{d}^j|
	=
	\binom{n + j - 1}{n}
\]
as any element of $\bm{d}^j$ corresponds to a non-negative integer solution of
$y_1 + y_2 + \ldots + y_n = j$ where $y_i$ is the count of $x_i$ in a given
$d^j \in \bm{d}^{j}$

\subsection{$\mathcal{D}_{\mal}(A)$ When $A$ Has Codimension $1$}

Let $A \subset \Kx$ be a subalgebra of codimension $1$. Our goal is to find a
basis for $\mathcal{D}_{\mal}(A)$. From Theorem \ref{th-gorin}, we know that
$A$ is the kernel of some $\mal, \mbe$-subalgebra condition $L : \Kx \to \K$,
and we split into cases depending on if $L$ is an $\mal$-derivation or an
$\mal, \mbe$-character difference. \\

We begin by considering the case when $\mal = \mbe$ and $L$ is an
$\mal$-derivation. We denote the entries in $\mal$ by $\mal = (\bm{\alpha}_{1},
\ldots, \alpha_{n})$. By Corollary \ref{cor-trivial_derivations}, we know that
$L$ is of the form $L : f \mapsto f'_{\bm{u}}(\mal)$ for some $\bm{u} = (u_{1},
\ldots, u_n)$. We know from Corollary \ref{cor-deriv_bound} that
$\dim(\mathcal{D}_{\mal}(A)) \leq 2n$. Moreover, as $A$ has codimension $1$ in
$\Kx$, it follows that any set of $2n + 1$ linearly independent linear
functionals on $\Kx$ will span $\mathcal{D}_{\mal}(A)$ if they restrict to
$\mal$-derivations on $A$, and this is the approach we will take to find a
spanning set. \\

Let $j \in [1..n]$ be such that $u_j \not = 0$. We define $2n + 1$ linear functionals on $\Kx$ according to
\begin{align*}
	D_{x_i} & : f \mapsto f'_{x_i}(\mal), \, i \in [1..n]              \\
	D_{x_i \bm{u}} & : f \mapsto f''_{x_i \bm{u}}(\mal), \, i \in [1..n] \\
	D_{\bm{u}\bm{u}\bm{u}} & : f \mapsto f'''_{\bm{u} \bm{u} \bm{u}}(\mal).
\end{align*}
After suitable rescaling, these linear functionals form a dual basis to the
subspace of $\Kx$ spanned by
\begin{align*}
	 & (x_i - \alpha_i), \, i \in [1..n]                 \\
	 & (x_i - \alpha_i)(x_j - \alpha_j), \, i \in [1..n] \\
	 & (x_j - \alpha_j)^{3}.
\end{align*}
In particular, the linear functionals are linearly independent on $\Kx$. We now
verify that they become $\mal$-derivations when restricted to $A$. \\

As the $D_{x_i}$ are $\mal$-derivations on $\Kx$, it follows that they remain
$\mal$-derivations when restricted to $A$. To see that the $D_{x_i\bm{u}}$
become $\mal$-derivations when restricted to $a$, note that for $f, g \in A$ we
have
\begin{align*}
	(fg)''_{x_i \bm{u}} 
	 & =
	f''_{x_i\bm{u}}(\mal) g(\mal)
	+
	f'_{x_i}(\mal) g'_{\bm{u}}(\mal)
	+
	f'_{\bm{u}}(\mal) g'_{x_i}(\mal)
	+
	f(\mal) g''_{x_i\bm{u}}(\mal) \\
     &=
	f''_{x_i\bm{u}}(\mal) g(\mal)
	+
	f(\mal) g''_{x_i\bm{u}}(\mal)
\end{align*}
since
\[
	f'_{x_i}(\mal) g'_{\bm{u}}(\mal)
	= 
	f'_{\bm{u}}(\mal) g'_{x_i}(\mal)
	= 
	0 
\]
for $f, g \in A$. Similarly, $D_{\bm{u}\bm{u}\bm{u}}$ becomes an
$\mal$-derivation when restricted to $A$ as 
\begin{align*}
	(fg)'''_{\bm{u} \bm{u} \bm{u}}(\mal)
	 & =
	f'''_{\bm{u} \bm{u} \bm{u}}(\mal)
	g
	+
	3 f''_{\bm{u} \bm{u}}(\mal)
	g'_{\bm{u}}(\mal)
	+
	3 f'_{\bm{u}}(\mal)
	g''_{\bm{u} \bm{u}}(\mal)
	+
	f
	g'''_{\bm{u} \bm{u} \bm{u}}(\mal) \\
     &=
	f'''_{\bm{u} \bm{u} \bm{u}}(\mal)
    +
	g'''_{\bm{u} \bm{u} \bm{u}}(\mal).
\end{align*}
We have indeed constructing a spanning set for $\mathcal{D}_{\mal}(A)$, and it
is easy to see that we may obtain a basis by simply excluding some $D_{x_i}$
for an index $i$ with $u_i \not = 0$. \\

Now let $B = \Kx \cap \Ker(L)$ where $L : \Kx \to \K$ is given by $L : f
\mapsto f(\mal) - f(\mbe)$. Applying Theorem
\ref{th-connect_clusters_direct_sum_deriv_space} yields
\[
	\mathcal{D}_{\mal}(B)
	=
	\restr{
		\mathcal{D}_{\mal}(\Kx)
	}{B}
	\oplus
	\restr{
		\mathcal{D}_{\mbe}(\Kx)
	}{B},
\]
after which we can use Lemma \ref{lm-Kx_derivations} to see that
$\mathcal{D}_{\mal}(B)$ is a $2n$-dimensional space consisting of functionals
of the form $f \mapsto a f'_{\bm{u}}(\mal) + b f'_{\bm{v}}(\mbe)$.

\subsection{Main Theorem of $\mal$-Derivations}

Extrapolating from the univariate case treated in
\cite{gronkvist2022subalgebras,leffler2021kandidat}, and the exploration of the
codimension $1$ situation, it seems reasonable to hypothesize the following
Theorem (which we will indeed prove in this section).

\begin{theorem}[Main Theorem of $\mal$-Derivations]\label{th-main}
	Let $A \subset \Kx$ be a subalgebra of finite codimension. Then there
	exist some integer $N$ such that any $\mal$-derivation over $A$ can 
	be written in the form
	\begin{equation}\label{eq-main_th}
		f
		\rightarrow 
		\sum_{\mal_i \sim \mal}
		\sum_{j = 1}^{N - 1}
		\sum_{d \in \bm{d}^j}
		c_{i,j,d} f^{(j)}_{d} (\mal_i)
	\end{equation}
	where each $c_{i,j,d} \in \K$. 
\end{theorem}

Throughout this section, we will assume that $A$ only has a single cluster.
This will simplify our efforts, and we can easily recover the general case
using Theorem \ref{th-conditions_on_other_clusters}. \\ 

We shall require a few constructions before giving a proof of Theorem
\ref{th-main}. We begin by defining two sets of linear functionals, and show
that their union is a set of subalgebra conditions.
\begin{definition}
	Let $S \subset \K^n$ be a finite set, and $\mal$ be some element in
	$S$. We define $\mathcal{E}(S, \mal)$ to be the set of character differences
	\[
		\mathcal{E}(S, \mal)
		=
		\{
		f \rightarrow f(\mal) - f(\mbe)
		:
		\mbe \in S \setminus \{\mal\}
		\}.
	\]
    The specific choice of $\mal \in \Sp(A)$ will often be irrelevant to us,
    and in these cases we shall simply write $\mathcal{E}(\Sp(A))$.
\end{definition}

\begin{lemma}
    Let $A \subset \Kx$ be a subalgebra of finite codimension where the
    subalgebra spectrum consists of a single cluster, and $\mal \in \Sp(A)$.
    Then $\mathcal{E}(\Sp(A), \mal)$ is a basis for the space of character
    differences which vanish on $A$.
\end{lemma}
\begin{proof}
    First of, suppose that $L$ is a $\mga, \mde$-character difference vanishing
    on $A$. Then as $\mal, \mga, \mde \in \Sp(A)$, we have $L_{1} = f
    \rightarrow f(\mal) - f(\mga), L_{2} = f \rightarrow f(\mal) - f(\mga) \in
    \mathcal{E}(\Sp(A), \mal)$ and $L$ is a scalar multiple of $L_{1} - L_{2}$.
    \\ 

    Moreover, it follows from Corollary \ref{cor-merge_two_clusters} that we
    need kernel $\Kx$ by $\vert \Sp(A) \vert - 1$ different character
    differences to merge all the elements of $\Sp(A)$ into a single cluster.
    Thus the space of character differences vanishing on $A$ has dimension at
    least $\vert \Sp(A) \vert - 1$, which is the cardinality of
    $\mathcal{E}(\Sp(A), \mal)$.
    
\end{proof}

\begin{definition}
    When $N > 1 \in \N$, we define $\mathcal{D}_N(S)$ to be the set of linear
    functionals
	\[
		\mathcal{D}_N(S)
		=
		\left\{
		f 
		\rightarrow
		f^{(j)}_{d} (\mal)
		:
		\mal \in S,
		d \in \bm{d}^j,
		j \in [1..N-1]
		\right\}.
	\]
\end{definition}
Lastly we combine these to sets to define a special subalgebra.
\begin{definition}
    Let $S \subset \K^{n}$ be a finite set of spectral elements and $N > 1 \in
    \N$. We $\K$-space $Q_N(S)$ to be given by
    \[
		Q_N(S)
		=
		\bigcap_{L \in \mathcal{E}(S) \cup \mathcal{D}_N(S)}
		\Ker(L).
	\]
\end{definition}

\begin{lemma}
    $Q_N(S)$ is a $\K$-algebra.
\end{lemma}
\begin{proof}
    We prove this by induction on $N$. The base case when $N = 1$ is immediate
    since the elements of $\mathcal{E}(S) \cup \mathcal{D}_{1}(S)$ are
    subalgebra conditions on $\Kx$. \\ 

    Suppose now that $Q_{N - 1}$ is a $\K$-algebra. Every element in $D_N \in
    \mathcal{D}_{N}(S) \setminus \mathcal{D}_{N - 1}(S)$ becomes a subalgebra
    condition when restricted to $Q_{N - 1}(S)$, since the middle terms found
    by expanding $D_N$ via the General Leibniz rule vanish on $A$. As $Q_{N}(S)
    = Q_{N - 1}(S) \cap \ker(\mathcal{D}_{N}(S) \setminus \mathcal{D}_{N -
    1}(S))$, it follows that $Q_N(S)$ is a subalgebra of $Q_{N - 1}(S)$ and we
    are done.
    
\end{proof}

As we shall see, the Main Theorem follows fairly easily as long as we can show
that $A$ contains some $Q_N(\Sp(A))$, and that the derivations over any $Q_N$
can be written as in Theorem \ref{th-main}. 

\begin{lemma}\label{lm-A_contains_Q}
	Let $A \subset \Kx$ be a single cluster subalgebra of finite codimension.
	Then there exists some $N \in \N$ such that 
	\[
		Q_{N}(\Sp(A)) \subseteq A.
	\]
\end{lemma}

\begin{lemma}\label{lm-D_of_Q}
    Let $S \subset \K^{n}$ be a finite set of spectral elements, $\mal \in S$,
    and $N \in \N$. Then	
	\[
		\mathcal{D}_{\mal}(Q_N(S))
		=
		\langle
		\mathcal{D}_{2N}(S)
		\setminus
		\mathcal{D}_{N}(S)
		\rangle.
	\]
\end{lemma}

We proceed by giving a proof Theorem \ref{th-main} using the previous two
lemmas.

\begin{proof}[Proof of The Main Theorem of $\mal$-Derivations.]
	
    Consider first the case when $A$ is a single cluster subalgebra. Let $\mal
    \in \Sp(A)$ and $D$ be an arbitrary $\mal$-derivation on $A$. By Lemma
    \ref{lm-A_contains_Q}, there exists some $N'$ such that $Q_{N'}(\Sp(A))
    \subseteq A$. Moreover, we know from Lemma \ref{lm-D_of_Q} that
    $\mathcal{D}_{\mal}(Q_{N'}(\Sp(A))) \subset \langle
    \mathcal{D}_{2N'}(\Sp(A)) \rangle$. Now, as $D$ must be an
    $\mal$-derivation when restricted to $Q_{N'}(\Sp(A))$ it follows that
    $\restr{D}{Q_{N'}(\Sp(A))}$ lies in $\langle \mathcal{D}_{2N'}(\Sp(A))
    \rangle$ and can be written as in equation (\ref{eq-main_th}) with $N = 2N'$.
    Moreover, it follows from the construction of $Q_{N'}(\Sp(A))$ that
    $Q_{N'}(\Sp(A))$ may be obtained from $A$ as the intersections of kernels
    of linear functionals. These may also be written as in equation
    (\ref{eq-main_th}), whence it follows from Lemma \ref{lm-intersect_lin_funcs}
    that $D$ can be written as in equation (\ref{eq-main_th}) as well. \\
	
    Generalization to arbitrary subalgebras $A \subset \Kx$ of finite
    codimension (i.e, not necessarily single cluster), follows using Theorem
    \ref{th-conditions_on_other_clusters}: Let $C \subset \Sp(A)$ be some
    cluster of $A$. Let $\mathcal{L}$ be the set of character differences and
    derivations that describe $A$ which pertain to $C$, and let $A' =
    \bigcap_{L \in \mathcal{L}} \Ker(L)$. Let $\mal \in C$. Then
    $\mathcal{D}_{\mal}(A) = \restr{\mathcal{D}_{\mal}(A')}{A}$ by Theorem
    \ref{th-conditions_on_other_clusters}, and as all $\mal$-derivations in
    $\mathcal{D}_{\mal}(A')$ may be written as in equation (\ref{eq-main_th})
    (assuming we proved the Main Theorem for single cluster subalgebras), they
    can also be written in this desired way when restricted to $A$ as well.

\end{proof}

The remainder of this section will be devoted to proving lemmas
\ref{lm-A_contains_Q} and \ref{lm-D_of_Q}. The route we will take is to first
define a subalgebra $Q'_{N}(\Sp(A))$ which is easily shown to be contained in
$A$ for some $N \in \N$. After this we will spend quite a bit of effort to show
that $Q'_N(\Sp(A)) = Q_N(\Sp(A))$ which will result in a proof of Lemma
\ref{lm-A_contains_Q}. We will obtain a proof of Lemma \ref{lm-D_of_Q} along
the way as well.

\begin{definition}
    Let $\mal \in \K^{n}$, $P(\mal) = \{x_i - \alpha_i : \alpha_i \in \mal\}$,
    and $P_N(\mal)$ be the subset of $P(\mal)_{\text{mon}}$ consisting of all
    monomials in elements of $P(\mal)$ of total degree $N$. For example,
	\[
		P_3(2, 1)
		=
		\{
		(x_1 - 2)^3,
		(x_1 - 2)^2(x_2 - 1),
		(x_1 - 2)(x_2 - 1)^2,
		(x_2 - 1)^3
		\}.
	\]
	If $|\mal| = n$, then
	\[
		|P_N(\mal)| = \binom{n + N - 1}{N}.
	\]
    Elaborating further, if $S \subset \K^{n}$ is a set of spectral elements we
    will use the notation $\Pi_N(S)$ to denote the set consisting of all
    possible product combinations of elements from the sets $P_N(\mal)$ for
    every $\mal \in S$. For example, if $S = \{(2,1), (0, 0), (1, 3)\}$, then
    every element of $\Pi_N(S)$ will be a product of one polynomial in
    $P_N(2,1)$, one polynomial in $P_N(0, 0)$, and one polynomial in $P_N(1,
    3)$, and every possible combination of polynomials drawn from the three
    sets, will exist as a product in $\Pi_N(S)$. I.e, the set will have
    cardinality 
	\[
		\vert \Pi_N(S) \vert
		= 
        \binom{n + N - 1}{N}^{\vert S \vert},
	\]
    and all polynomials in $\Pi_N(S)$ will have total degree $Ns$. A quick
    example is given by
	\begin{align*}
		\Pi_2(\{(0, 0), (0, 1)\})
		= & 
		\{ 
		x_1^4, x_1^3 (x_2 - 1), x_1^2 (x_2 - 1)^2,                \\
		  & x_1^3 x_2, x_1^2 x_2 (x_2 - 1), x_1 x_2 (x_2 - 1)^2   \\
		  & x_2^2 x_1^2, x_2^2 (x_2 - 1) x_1, x_2^2 (x_2 - 1)^2 
		\}.
	\end{align*}
	
	Let $\Pi_N(S) \Kx$ be the ideal generated by $\Pi_N(S)$ in $\Kx$. We define
	$Q'_N$ to be the subalgebra 
	\[
		Q'_N(S)
		=
        \K
        \oplus
		\Pi_N(S) \Kx.
	\]
\end{definition}

\begin{lemma}\label{lm-D_N_vanish_on_Qprim_N}
	Let $f \in \Kx, \pi \in P_N(\mal)$ and $D \in
		\mathcal{D}_{N}(\{\mal\})$. Then $D(f\pi) = 0$.
\end{lemma}
\begin{proof}
    By the construction of $\mathcal{D}_{N}(\{\mal\})$, $D$ is given as $D(f) =
    f^{(j)}_{d}(\mal)$ for some $d \in \bm{d}^{j}$ and $j \in [1.. N - 1]$. By the
    generalized Leibniz rule we have 
	\begin{align*}
		D(f \pi)
		 & =
		\sum_{d' \in \mathcal{P}(d)}
		f^{(j - \vert d'\vert)}_{d \setminus d'}(\mal)
		\pi^{(\vert d' \vert)}_{d'}(\mal),
	\end{align*}
    from which we obtain $D(f \pi) = 0$ using the fact that $\pi^{(\vert d'
    \vert)}_{d'}(\mal) = 0$ for all $d'$. Indeed, $\pi$ is a product of $N$
    linear factors which vanish on $\mal$, hence as $\vert d' \vert < N$, any
    $\vert d' \vert$-th derivative of $\pi$ will be a sum where each term is
    divisible by at least one such factor. 
    
\end{proof}

It follows that $Q'_N \subset Q_N$.

\begin{corollary}\label{cor-Qprime_in_Q}
    Let $S \subset \K^{n}$ be a set of spectral elements. Then $Q'_N(S)
    \subseteq Q_N(S)$.
\end{corollary}
\begin{proof}
	It is trivial to see that $\mathcal{E}(S)$ kills all of $Q'_N(S)$ since the
	generators of $Q'_N(s)$ all have roots in every element of $S$. It follows
	from the prior lemma that all elements in $\mathcal{D}_N(S)$ vanish on
	$Q'_N(s)$.
	
\end{proof}

We now show that some $Q'_N$ is contained in $A$.

\begin{lemma}\label{lm-ideal_subalg_in_A}
	Let $A \subset \Kx$ be a single cluster subalgebra of finite codimension.
	Then there exists some $N \in \N$ such that $Q'_N(\Sp(A)) \subseteq A$.
\end{lemma}
\begin{proof}
	We will prove the lemma by induction on the codimension of $A$. For our
	induction step, we will only consider kerneling by
	$\mal$-derivations. We can do this since we can kernel by all character
	differences before we start kerneling by $\mal$-derivations. This
	will require our base case to treat a subalgebra obtained by kerneling 
	with any amount of character differences. \\
	
	Consider the case of a single cluster subalgebra $A$ which is obtained from
	$\Kx$ by kerneling by character differences only. Let $N = 1$ and note that
	any $\pi \in \Pi_1(\Sp(A))$ has a root in every spectral element. Hence
	$E(f\pi) = 0$ for all $\pi \in \Pi_1(\Sp(A)), f \in \Kx$ and any character
	difference $E$ that holds in $A$. It follows that $Q'_1(\Sp(A)) \subseteq
		A$. \\
	
	Moving on to the induction step, let $A' \subseteq \Kx$ be a single cluster
	subalgebra of finite codimension such that $Q'_{N'}(\Sp(A')) \subseteq A'$
	for some $N' \in \N$. Let $A$ be obtained from $A'$ as the kernel of some
	non-trivial $\mal$-derivation $D$. Note that $\Sp(A) = \Sp(A')$ as $D$ is
	assumed to be a non-trivial $\mal$-derivation. We set $N = 2N'$. For any
	$\mbe \in \Sp(A) = \Sp(A')$, we have that each polynomial in $P_N(\mbe)$
	can be written as a product of two polynomials in $P_{N'}(\mbe)$. Thus each
	$\pi \in \Pi_N(\Sp(A))$ can be written $\pi = \pi_1 \pi_2$ for $\pi_1,
		\pi_2 \in \Pi_{N'}(\Sp(A'))$. It follows that $f \pi  = (f\pi_1) (\pi_2)
		\in \mathfrak{m}_{\mal}(A')^2$ whence $D(f \pi) = 0$ and $f \pi \in A$ for
	all $\pi \in \Pi_N(\Sp(A)), f \in \Kx$. It follows that $Q'_N(\Sp(A))
		\subseteq A$. \\

\end{proof}

We now show that $Q'_N(S) = Q_N(S)$ in the case when $S = \{\mal\}$ consists of
a single element.

\begin{lemma}
	Let $\mal \in \K^n$. Then $Q'_N(\{\mal\}) = Q_N(\{\mal\})$.
\end{lemma}
\begin{proof}
    Let a vector space basis $T$ be given for $\Kx$ as 
    \[
        T = 
        \left\{
                (x_{1} - \alpha_{1})^{a_{1}}
                \ldots
                (x_{n} - \alpha_{n})^{a_{n}}
            :
            \bm{a} = (a_{1}, \ldots, a_n) \in \N^{n}
        \right\}.
    \]
    Then $Q'_N(\{\mal\})$ is generated by the subset $T_N \subset T$ consisting
    of all polynomials in $T$ of total degree $\geq N$. Meanwhile, the elements
    of $\mathcal{D}_{N}(\{\mal\})$ vanishes on all of the monomials in $T_N$,
    and after suitable rescaling, form a dual basis to its complement $T
    \setminus T_N$. It follows that $Q'_N(\{\mal\}) =
    \ker(\mathcal{D}_{N}(\{\mal\})) = Q_{N}(\{\mal\})$ and we are done.
    
\end{proof}

We are now ready to prove Lemma \ref{lm-D_of_Q}.

\begin{proof}[Proof of Lemma \ref{lm-D_of_Q}]
	Our first aim is to show that the statement of the lemma holds when $S =
		\{\mal\}$, i.e that 
	\begin{equation}\label{eq-D_of_Q_singl_spec}
		\mathcal{D}_{\mal}(Q_N(\{\mal\}))
		=
		\langle
		\mathcal{D}_{2N}(\{\mal\})
		\setminus
		\mathcal{D}_{N}(\{\mal\})
		\rangle.
	\end{equation}
    We already know that $Q_N(\{\mal\}) = Q'_N(\{\mal\})$, and we will show
    that
    \[
        \mathcal{D}_{\mal}(Q'_N(\{\mal\}))
        =
        \langle
        \mathcal{D}_{2N}(\{\mal\})
        \setminus
        \mathcal{D}_{N}(\{\mal\})
        \rangle.
    \]
    We simplify by performing a change of variables so that $\mal \mapsto
    \bm{0}$. In this setting, a minimal SAGBI basis for $Q'_N(\{\bm{0}\})$ is
    given by all monomials of total degrees between $N$ and $2N-1$ inclusively.
    There are 
	\[
		\binom{n + 2N - 1}{n} - \binom{n + N - 1}{n}
	\]
	such monomials and it follows from Theorem \ref{th-dim_D} that 
	\[
		\dim(\mathcal{D}_{\mal}(Q'_N(\{\bm{0}\})))
		\leq 
		\binom{n + 2N - 1}{n} - \binom{n + N - 1}{n}.
	\]
	We also have that 
	\[
		|
		\mathcal{D}_{2N}(\{\bm{0}\})
		\setminus
		\mathcal{D}_{N}(\{\bm{0}\})
		|
		=
		\binom{n + 2N - 1}{n} - \binom{n + N - 1}{n},
	\]
    so equality in equation (\ref{eq-D_of_Q_singl_spec}) would follow in the $\mal =
    \bm{0}$ case if we could show that the linear functionals in
    $\mathcal{D}_{2N}(\{\bm{0}\}) \setminus \mathcal{D}_{N}(\{\bm{0}\})$ are
    linearly independent when restricted to $Q'_N(\bm{0})$ and that they define
    $\bm{0}$-derivations here. \\ 

    For linear independence, it suffices to note that
    $\mathcal{D}_{2N}(\{\bm{0}\}) \setminus \mathcal{D}_{N}(\{\bm{0}\})$ forms
    a dual space to the subspace of $Q'_N(\bm{0})$ spanned by the monomials
    $x^{\bm{a}}$ with total degree $\vert \bm{a} \vert \in [N..2N-1]$. \\
    
    Now we verify that the elements in $\mathcal{D}_{2N}(\{\bm{0}\}) \setminus
    \mathcal{D}_{N}(\{\bm{0}\})$ define $\bm{0}$-derivations on $Q'_N(\bm{0})$.
    Let $D : f \to
    f^{\vert d \vert}_{d}(\bm{0})$ with $\vert d \vert \in [N .. 2N - 1]$ be
    some element in $\mathcal{D}_{2N}(\{\bm{0}\}) \setminus
    \mathcal{D}_{N}(\{\bm{0}\})$. The General Leibniz rule says that
    \[
        D(fg) = 
        \sum_{d' \in \mathcal{P}(d)}
        f^{(|d'|)}_{d'}(\bm{0})
        g^{(|d| - |d'|)}_{d \setminus d'}(\bm{0}).
    \]
    Now, we already know that $Q'_N(\bm{0}) = Q_N(\bm{0}) =
    \ker(\mathcal{D}_{N}(\{\bm{0}\}))$, and if we inspect the Leibniz expansion
    above when $f, g \in Q'_N(\bm{0})$, we see that all middle terms vanish
    since either $\vert d \vert < N$ or $\vert d' \vert < N$. Hence $D$ is a
    derivation on $Q'_N(\bm{0})$ and we 
	\[
		\mathcal{D}_{\bm{0}}(Q'_N(\{\bm{0}\}))	
		=
		\langle
		\mathcal{D}_{2N}(\{\bm{0}\})
		\setminus
		\mathcal{D}_{N}(\{\bm{0}\})
		\rangle.
	\]
    Undoing our change of variables yields
	\[
		\mathcal{D}_{\mal}(Q'_N(\{\mal\}))	
		=
		\mathcal{D}_{\mal}(Q_N(\{\mal\}))	
		=
		\langle
		\mathcal{D}_{2N}(\{\mal\})
		\setminus
		\mathcal{D}_{N}(\{\mal\})
		\rangle.
	\]
    By Theorem \ref{th-conditions_on_other_clusters}, if we have two 
    distinct elements $\mal_i, \mal_j \in S$ then 
    \[
        \mathcal{D}_{\mal_i}\left(Q_N(\{\mal_i\}) \cap Q_N(\{\mal_j\})\right)
        =
        \mathcal{D}_{\mal_i}(Q_N(\{\mal_i\})),
    \]
    after which induction over the elements of $S$ yields
	\[
		\mathcal{D}_{\mal_i}\left(
		\bigcap_{\mal \in S} Q_N(\{\mal\})
		\right)
        =
        \mathcal{D}_{\mal_i}(Q_N(\{\mal_i\}))
		=
		\langle
		\mathcal{D}_{2N}(\{\mal_i\})
		\setminus
		\mathcal{D}_{N}(\{\mal_i\})
		\rangle,
	\]
    Now we may use Theorem \ref{th-connect_clusters_direct_sum_deriv_space}
    inductively to obtain
	\[
		\mathcal{D}_{\mal}(Q_N(S))	
		=
		\bigoplus_{\mal \in S}
		\langle
		\mathcal{D}_{2N}(\{\mal\})
		\setminus
		\mathcal{D}_{N}(\{\mal\})
		\rangle
		=
		\langle
		\mathcal{D}_{2N}(S)
		\setminus
		\mathcal{D}_{N}(S)
		\rangle
	\]
	which completes our proof.	
	
\end{proof}

We are now ready to prove that $Q_N = Q'_N$ in general.

\begin{lemma}\label{lm-Qprime=Q}
	Let $S \subset \K^n$ be a finite set of points and $N \in \N$. Then
	$Q'_N(S) = Q_N(S)$.	
\end{lemma}
\begin{proof}
	We already know from Corollary \ref{cor-Qprime_in_Q} that $Q'_N(S) \subset
		Q_N(S)$. From Theorem \ref{th-gorin}, it follows that
	\[
		Q'_N(S)
		=
		Q_N(S) \cap \bigcap_{L \in \mathcal{L}} L,
	\]
	for some set of subalgebra conditions $\mathcal{L}$. We will prove the
	theorem by showing that $\mathcal{L}$ is empty, which in turn will be
	proven by showing that no non-zero subalgebra conditions over $Q_N(S)$
	vanishes on $Q'_N(S)$. We have three kinds of subalgebra conditions to
	consider,
	\begin{itemize}
		\item Non-trivial $\mal$-derivations, i.e elements of
		      \[
			      \mathcal{D}_{\mal}(Q_N(\Sp(A))) 
			      = 
			      \mathcal{D}_{2N}(\{\mal\})
			      \setminus
			      \mathcal{D}_N(\{\mal\})
		      \]
		      for $\mal \in S$.
		\item Trivial $\mal$-derivations.
		\item Character differences.
	\end{itemize}
	
    We begin with the case when $D \in \mathcal{D}_{\mal}(Q_N(\Sp(A)))$ is a
    non-trivial $\mal$-derivation, which by Lemma \ref{lm-D_of_Q} is given as a
    linear combination of functions of the form
	\[
		D_{d}(f) = f^{(j)}_{d}(\mal)
	\]
    for $d \in \bm{d}^j$ and $j \in [N..2N-1]$. We will treat the case of
    $\mal$-derivations of the form $D_{d}$ first, after which obtaining the
    result for linear combinations of such $\mal$-derivations will follow
    easily. For a given $d \in \bm{d}^{j}$, we will construct a polynomial in
    $Q'_N(\Sp(A))$ on which the $\mal$-derivation $D_{d}$ is not zero. Let 
	\[
		\pi_1(\bm{x})
		=
		\prod_{x_i \in d} (x_i - \alpha_i).
	\]
	Then $\pi_1 = f \widehat{\pi}_1$ for some $\widehat{\pi}_1 \in
		P_N(\mal), f \in \Kx$ and also $D(\pi_1) = M$ for some integer $M$. For
	every spectral element in $\mbe \in \Sp(A) \setminus \{\mal\}$
	there must exist some index $i(\mbe)$ such that $\beta_{i(\mbe)}
		\not = \alpha_{i(\mbe)}$. Construct 
	\[
		\pi_2(\bm{x})
		=
		\prod_{\mbe \in \Sp(A) \setminus \{\mal\}}
		(x_{i(\mbe)} - \beta_{i(\mbe)})^N.
	\]
	Then $\pi_2(\mal) \not = 0$ and $\pi = \pi_1 \pi_2 \in \Pi_N(\Sp(A))
		\subset Q'_N(\Sp(A))$. Applying our derivation, we get 
	\begin{align*}
		D(\pi)
		 & =
		(\pi_1 \pi_2)^{(j)}_{d}(\mal) \\
		 & =
		(\pi_1)^{(j)}_{d}(\mal)
		\pi_2(\mal)
		+
		\pi_1(\mal)
		(\pi_2)^{(j)}_{d}(\mal)       \\
		 & =
		M	
		\pi_2(\mal)                       \\
		 & \not =
		0.
	\end{align*}
	For arbitrary $\mal$-derivations, we can combine linear combinations
	of polynomials like $\pi$ above, one for each derivative evaluation, and
	since $\K$ is infinite, we can choose scalars in such a way that the given
	$\mal$-derivation does not vanish on the linear combination of
	polynomials. \\
	
    To see that no extra trivial derivation belongs to $\mathcal{L}$, let $\mga
    \not \in \Sp(A)$, and consider a $\mga$-derivation of the form $D_i : f
    \mapsto f'_{x_i}(\mga)$. For every spectral element $\mal \in \Sp(A)$,
    there is an index $i(\mal)$ such that $\mal_{i(\mal)} \not =
    \mga_{i(\mal)}$. Now we can construct a polynomial
    \[
        \pi_3 =  
		\prod_{\mal \in \Sp(A)}
		(x_{i(\mal)} - \mal_{i(\mal)})^N \in \Pi_{N}(\Sp(A))
    \]
    such that $\pi_3(\mga) \not = 0$. Then if we multiply $\pi_3$ by $x_i -
    \gamma_i$ we get that $D((x_i - \gamma_i)\pi_3) = \pi_3(\mga) \not = 0$.
    We extend the result to deal with arbitrary $\mga$-derivations in exactly
    the same way as we did for non-trivial $\mal$-derivations above. \\
	
    To see that no character difference belongs to $\mathcal{L}$, let $\mga
    \not \in \Sp(A)$ and construct $\pi_3$ like above. Let $\mde \in \K^n$ such
    that $\mde \not = \mga$. Then there exists some index $i$ such that
    $\delta_i \not = \gamma_i$. Let $h = (x_i - \delta_i)\pi_3 \in
    Q'_N(\Sp(A))$. Then $h(\mga) - h(\mde) = (\gamma_i - \delta_i)h(\mga) \not
    = 0$, and $\mga \not \sim \mde$ in $Q'_N(\Sp(A))$. 
    
\end{proof}

Lemma \ref{lm-A_contains_Q} now follows trivially,
\begin{proof}[Proof of Lemma \ref{lm-A_contains_Q}]
	Follows immediately from Lemmas \ref{lm-ideal_subalg_in_A} and
	\ref{lm-Qprime=Q},
	
\end{proof}
and we can finally consider the Main Theorem as settled.

\section{Acknowledgements}
	I want to express my deepest gratitude to my supervisors Anna Torstensson
	and Victor Ufnarovski for all their guidance and support. This paper would
	not have been possible without their help. 

\printbibliography

\end{document}